\documentclass{amsart}

\usepackage{enumitem}
\usepackage{amssymb}
\usepackage{amsthm}
\usepackage{amsmath}
\usepackage{hyperref}

\theoremstyle{plain}
\newtheorem{proposition}{Proposition}[section]
\newtheorem{theorem}[proposition]{Theorem}
\newtheorem{lemma}[proposition]{Lemma}
\newtheorem{corollary}[proposition]{Corollary}

\theoremstyle{definition}
\newtheorem{example}[proposition]{Example}
\newtheorem{definition}[proposition]{Definition}
\newtheorem{observation}[proposition]{Observation}

\theoremstyle{remark}
\newtheorem{remark}[proposition]{Remark}
\newtheorem{notation}[proposition]{Notation}

\numberwithin{equation}{section}

\DeclareMathOperator{\diam}{diam}

\DeclareMathOperator{\Cay}{Cay}

\DeclareMathOperator{\id}{id}

\newcommand{\vertiii}[1]{{\left\vert\kern-0.25ex\left\vert\kern-0.25ex\left\vert #1 
    \right\vert\kern-0.25ex\right\vert\kern-0.25ex\right\vert}}

\begin{document}

\title{Anosov representations over closed subflows}
\author{Tianqi Wang}
\address{National University of Singapore}
\curraddr{}
\email{twang@u.nus.edu}
\thanks{Wang was partially supported by the NUS-MOE grant R-146-000-270-133}

\date{}

\dedicatory{}

\begin{abstract}
We introduce a generalization of the notion of Anosov representations by restricting to invariant closed geodesic subflows. Examples of such representations include many non-discrete representations with good geometric properties, such as primitive-stable representations. We give several equivalent characterizations of this type of representations and prove some properties analogous to the classical Anosov representations, such as stability, the Cartan property and regularity properties of the limit maps.
\end{abstract}

\maketitle

\section{Introduction}
When $G$ is a semisimple Lie group of rank $1$, we say that a representation $\rho$ of a finitely generated group $\Gamma$ into $G$ is \emph{convex-cocompact} if it acts properly discontinuously and cocompactly on some nonempty closed convex subset of the symmetric space of $G$. Equivalently, $\rho$ is a convex-cocompact representation if and only if the orbit map from $\Gamma$ to the symmetric space of $G$, induced by $\rho$, is a quasi-isometric embedding. Convex-cocompact representations are stable, which means the set of such representations is open in $\mathrm{Hom}(\Gamma,G)$. The notion of Anosov representations of a finitely generated group $\Gamma$ into a semisimple Lie group, introduced by Labourie \cite{Lab}, generalizes this notion to higher rank Lie groups. Anosov representations are also stable and induce quasi-isometric embeddings of $\Gamma$ to the symmetric spaces of the Lie groups. Since then, the theory of Anosov representations was developed further by many other researchers, such as Guichard--Wienhard \cite{GW}, Kapovich--Leeb--Porti \cite{KLP1}, \cite{KLP2}, \cite{KLP3}, Gu\'eritaud--Guichard--Kassel--Wienhard \cite{GGKW}, and Bochi--Potrie--Sambarino \cite{BPS}.\\

There are many extensions of the notion of Anosov representations, for example,  Kapovich and Leeb's work on relatively Morse and relatively asymptotically embedded representations \cite{KL} and relatively dominated representations defined by Zhu \cite{Zhu}. A common feature of these notions is that all these types of representations are discrete with finite kernels.\\

On the other hand, there are many representations which are not discrete, but still have good geometric properties. For example, Minsky \cite{Min} introduced the concept of \emph{primitive-stable representations}. They are representations of the free group $F_n$ on $n$ generators into $\mathrm{PSL}(2,\mathbb{C})$, where the induced orbit maps from the Cayley graphs to $\mathbb{H}^3$ restricted to primitive geodesics, are quasi-isometric embedding. We view the set $\mathcal{PS}(F_n,\mathrm{PSL}(2,\mathbb{C}))$ of the conjugacy classes of primitive-stable representations as a subset of the character variety $\chi(F_n, \mathrm{PSL}(2,\mathbb{C}))$. There is a natural $\mathrm{Out}(F_n)$-action on $\chi(F_n, \mathrm{PSL}(2,\mathbb{C}))$ by pre-composition since the set of primitive elements in $F_n$ is $\mathrm{Out}(F_n)$-invariant. Minsky \cite{Min} proved that $\mathcal{PS}(F_n,\mathrm{PSL}(2,\mathbb{C}))$ is open in $\chi(F_n, \mathrm{PSL}(2,\mathbb{C}))$, strictly larger than the set of the conjugacy classes of convex-cocompact representations, and $\mathrm{Out}(F_n)$ acts on $\mathcal{PS}(F_n,\mathrm{PSL}(2,\mathbb{C}))$ properly discontinuously.\\

Guéritaud--Guichard--Kassel--Wienhard \cite{GGKW} generalized the notion of primitive-stable  representations of free groups into higher a rank semisimple Lie groups $G$ (see Definition \ref{defps}). In this generalized case, a recent work by Kim--Kim \cite{KK} showed that the set of primitive-stable representations is open in $\chi(F_n, G)$, and the $\mathrm{Out}(F_n)$ action by pre-composition is properly discontinuous. The proof is a generalization of Minsky's proof, together with an application of a local to global theorem for Morse quasi-geodesics from Kapovich--Leeb--Porti \cite{KLP0}.\\

The notion of \emph{directed Anosov representations} introduced by Kim--Tan--Zhang \cite{KTZ} is also an example of non-discrete, but geometrically well-behaved representations. They used such representations to construct some new examples of primitive-stable representations. Informally, a directed Anosov representation of a hyperbolic group $\Gamma$ into a semisimple Lie group $G$ requires that it has the dominated property along some geodesic rays in $\Gamma$ (see Definition \ref{defda}).\\

Another type of examples are \emph{pleated surfaces},  which were first introduced by Thurston \cite{Thu}. Following from Bonahon \cite{Bonahon}, a pleated surface can be understood as a certain way to bend the universal covering of a hyperbolic surface $S$ in $\mathbb{H}^3$, along a maximal geodesic lamination $\lambda$, that is, a maximal collection of pairwise non-intersecting simple closed geodesics (see Section \ref{pleatedsurfaceeg}). Then a pleated surface gives a representation of the fundamental group $\pi_1(S)$ into $\mathrm{PSL}(2,\mathbb{C})$. Such representations are not necessarily discrete, but they are geometrically meaningful, because they contain a lot of information of the pleated surfaces.\\

The aim of this paper is to establish a common framework to study such non-discrete representations using tools from the theory of Anosov representations.\\

We assume that $\Gamma$ is a hyperbolic group with a finite, symmetric generating set $S$. If $\Gamma$ acts properly discontinuously and cocompactly on a topological space $X$ and commutes with a flow $\Phi$, then every linear representation of $\Gamma$ gives rise to a flat vector bundle over $X/\Gamma$ with a linear lift of $\Phi$ and one can define a corresponding dominated splitting condition. Anosov representations correspond to the case when $X$ is the geodesic flow of $\Gamma$. In this paper, we consider the more general case when $X$ is a closed, $\Gamma$-invariant subset of the geodesic flow.\\

More concretely, Gromov \cite{Gromov} (see also Mathéus \cite{Matheus}, Champetier \cite{Cham} and Mineyev \cite{Mineyev}) defined a properly discontinuous and cocompact $\Gamma$-action on $\partial^{(2)}\Gamma\times \mathbb{R}$, which commutes with the obvious $\mathbb{R}$-action on $\partial^{(2)}\Gamma\times \mathbb{R}$. One can then define  $U\Gamma=(\partial^{(2)}\Gamma\times \mathbb{R} )/\Gamma $, and the $\mathbb{R}$-action descends to a flow $\phi$ on $U\Gamma$. Let $P$ be a closed, $\Gamma$-invariant subset of $\partial^{(2)}\Gamma$. The geodesic subflow space in $U\Gamma$ associated to $P$ is given by $S_P= (P\times \mathbb{R})/\Gamma$, which is compact and $\phi$-invariant.\\

Now, let $\rho$ be a representation of $\Gamma$ into $\mathrm{GL}(d,\mathbb{R})$. The geodesic flow $\phi$ on $U\Gamma$ lifts to a parallel flow $\psi$ on the flat $\mathbb{R}$-vector bundle $E_\rho=U\Gamma\times_\rho \mathbb{R}^d$. We say that $\rho$ is \emph{$k$-Anosov over $S_P$} if there exist constants $C,\lambda>0$, such that $$\dfrac{\sigma_{d-k+1}(\psi^t_q)}{\sigma_{d-k}(\psi^t_q)}\leqslant Ce^{-\lambda t},$$ for any $q \in{S_P}$ and $t\in \mathbb{R}_+$, where $\sigma_i(\cdot)$ denotes the $i$-th singular value with respect to some (hence any) given Riemannian metric on $E_\rho$. In the case when $P=\partial^{(2)}\Gamma$, then $S_P= U\Gamma$, and we recover the classical definition of a $k$-Anosov representation as in \cite{BPS}.\\

We say that $E_\rho|_{S_P}=E^s_{\rho}\oplus E_\rho^u$ is a \emph{$k$-dominated splitting over $S_P$}, if $E_\rho^s$ and $E_\rho^u$ are $\psi$-invariant subbundles of $E_\rho|_{S_P}$, $\mathrm{rank}(E_\rho^s)=k$, and there exist constants $C,\lambda>0$, such that $$\frac{\Vert \psi_q^t(v)\Vert }{\Vert \psi_q^t(w)\Vert }\leqslant Ce^{-\lambda t},$$ for any $q\in S_P, t\in \mathbb{R}_{+}$ and $v\in E^s_{\rho,q}$, $w\in E^u_{\rho,q}$, with $\Vert v\Vert =\Vert w\Vert =1$.\\

On the other hand, let $\Cay(\Gamma,S)$ be the Cayley graph of $\Gamma$ with the generating set $S$. For a quasi-geodesic $l:\mathbb{R}\rightarrow \Cay(\Gamma, S)$, $l(+\infty)$ (respectively, $l(-\infty)$) denotes the forward (respectively, backward) endpoint of $l$. For constants $\lambda\geqslant 1$, $\epsilon\geqslant 0$ and $b\geqslant 0$, we consider the sets $$Q_{P,(\lambda,\epsilon,b)}=\{ l:\mathbb{R}\rightarrow \Cay(\Gamma,S)\ |\ l\ \text{is a}\ (\lambda,\epsilon)\text{-quasi-geodesic with}\ $$ $$ d(l(0),\id)\leqslant b\ \text{and}\ (l(+\infty),l(-\infty))\in P \}$$ and $\Gamma^+_{P,(\lambda,\epsilon,b)}=\{ l(t)\ |\ l\in Q_{P,(\lambda,\epsilon,b)}, t\in \mathbb{R}_+\ \text{with}\ l(t)\in \Gamma \}$. We also write $\Gamma^+_P=\Gamma^+_{P,(1,0,0)}$. We say that $\rho$ is \emph{$k$-dominated on $\Gamma_P^+$} if there exist constants $C,\lambda>0$, such that  $$\dfrac{\sigma_{k+1}(\rho(\gamma))}{\sigma_k(\rho(\gamma))}\leqslant Ce^{-\lambda|\gamma|},$$ for any $\gamma\in \Gamma_P^+$, where $\sigma_k(\cdot)$ denotes the $k$-th singular value with respect to the standard norm on $\mathbb{R}^d$. If $P=\partial^{(2)}\Gamma$, then $\Gamma_P^{+}=\Gamma$, and this specializes to the definition of $k$-dominated representations in \cite{BPS}.\\

Moreover, we generalize the definition of strongly dynamics preserving property in our case, which was introduced by Canary--Zhang--Zimmer \cite{CZZ} in the study of classical Anosov representations for geometrically finite Fuchsian groups. We denote $$\Pi_1(P)=\{ x\in\partial \Gamma\ |\ \exists y\in \partial \Gamma, \  (x,y)\in P\}\ \text{and}\  \Pi_2(P)=\{ y\in\partial \Gamma\ |\ \exists x\in \partial \Gamma, \ (x,y)\in P\}.$$ Let $\xi^k:\Pi_1(P)\rightarrow \mathrm{Gr}_k(\mathbb{R}^d)$ and $\xi_{d-k}:\Pi_2(P)\rightarrow \mathrm{Gr}_{d-k}(\mathbb{R}^d)$ be two $\rho$-equivariant maps such that $(\xi^k,\xi_{d-k}):P\rightarrow \mathrm{Gr}_k(\mathbb{R}^d)\times  \mathrm{Gr}_{d-k}(\mathbb{R}^d)$ is continuous and transverse. We say $(\xi^k,\xi_{d-k})$ is \emph{strongly dynamics preserving of index $k$ on $P$} if for any given constants $\lambda\geqslant 1$ and $\epsilon,b\geqslant 0$, points $x,y\in \partial\Gamma$, and a sequence $\{\gamma_n\}$ in $\Gamma^+_{P,(\lambda,\epsilon,b)}$, with $\gamma_n\rightarrow x$ and $\gamma_n^{-1}\rightarrow y$ as $n\rightarrow +\infty$, we have  $$\lim\limits_{n\rightarrow +\infty}\rho(\gamma_n)V=\xi^k(x)$$ for any $V\in \mathrm{Gr}_k(\mathbb{R}^d)$ transverse to $\xi_{d-k}(y)$.\\

Then we have the equivalence below.

\begin{theorem}\label{MT1}
    Let $\Gamma$ be a hyperbolic group and let $P$ be a closed, $\Gamma$-invariant subset of $\partial^{(2)}\Gamma$. Let $\rho$ be a representation of $\Gamma$ into $\mathrm{GL}(d,\mathbb{R})$.
    Then the following are equivalent.
    \begin{itemize}
        \item[(1).] $\rho$ is $k$-Anosov over $S_P$;
        \item[(2).] $E_\rho$ admits a $k$-dominated splitting over $S_P$;
        \item[(3).] $\rho$ is $k$-dominated on $\Gamma_P^+$;
        \item[(4).] There exists a unique pair of limit maps $\xi^k: \Pi_1(P)\rightarrow \mathrm{Gr}_k(\mathbb{R}^d)$ and $\xi_{d-k}:  \Pi_2(P)\rightarrow \mathrm{Gr}_{d-k}(\mathbb{R}^d)$, such that $(\xi^k,\xi_{d-k}): P\rightarrow \mathrm{Gr}_k(\mathbb{R}^d)\times \mathrm{Gr}_{d-k}(\mathbb{R}^d)$ is continuous, transverse and strongly dynamics preserving.
    \end{itemize}
    Moreover, the set of $k$-Anosov representations over $S_P$ is open in $\mathrm{Hom}(\Gamma,\mathrm{GL}(d,\mathbb{R}))$, i.e., $k$-Anosov representations over $S_P$ are stable.
\end{theorem}

When $\rho$ is $k$-dominated on $\Gamma_P^+$, we will investigate more about the properties of the limit maps $\xi^k$ and $\xi_{d-k}$. Unlike the classical Anosov representations, $\xi^k$ is not necessarily continuous on $\Pi_1(P)$ (see Example \ref{exampledisctslimitmap}). However, we may consider the sequences of sets $$Q^{+\infty}_{P,(\lambda,\epsilon,b)}= \{ l(+\infty)\in \partial\Gamma \ |\ l\in Q_{P,(\lambda,\epsilon,b)} \},$$ which satisfies $$\Pi_1(P)=\bigcup_{\lambda\geqslant 1, \epsilon\geqslant 0, b\geqslant 0} Q^{+\infty}_{P,(\lambda,\epsilon,b)}$$ (see Lemma \ref{unionconverge}). Then we will show the following properties.

\begin{theorem}\label{MT2}
    Let $\rho:\Gamma\rightarrow\mathrm{GL}(d,\mathbb{R})$ be a $k$-dominated representation on $\Gamma_P^+$.\\
        (1). (A weaker $P_k$-Cartan Property) Let $\gamma_n$ be a sequence in $\Gamma$ and $\gamma_n\rightarrow x\in \Pi_1(P)$ as $n\rightarrow \infty$. If we can find constants $\lambda\geqslant 1, \epsilon, b\geqslant 0$ such that  $\gamma_n\in \Gamma^+_{P,(\lambda,\epsilon,b)}$ for all $n\in \mathbb{N}$, then $$\xi^k(x)=\lim\limits_{n\rightarrow +\infty} U_k(\rho(\gamma_n)).$$
        (2). (Hölder continuity) For any constants $\lambda\geqslant 1$ and $\epsilon,b\geqslant 0$, $\xi^k$ is Hölder continuous on $Q^{+\infty}_{P,(\lambda,\epsilon,b)}$ and hence 
        $(\xi^k,\xi_{d-k})$ is Hölder continuous on any compact subset of $P$.
\end{theorem}

\begin{remark}\label{genLiegroup}
    In this paper, we only consider representations of hyperbolic groups into $\mathrm{GL}(d,\mathbb{R})$ as we can compute the singular values. The theory works if we replace $\mathrm{GL}(d,\mathbb{R})$ by $\mathrm{PGL}(d,\mathbb{R})$ since we can still compute the ratio of singular values. We can also replace $\mathbb{R}$ by $\mathbb{C}$ since $\mathrm{GL}(d,\mathbb{C})$ can be viewed as a subgroup of $\mathrm{GL}(2d,\mathbb{R})$. More generally, we can replace $\mathrm{GL}(d,\mathbb{R})$ by any non-compact semisimple Lie group with finite center. See \cite{CZZ} Appendix A.
\end{remark}

The paper is organized in the following order.
In Section \ref{preliminaries}, we mainly introduce some preliminaries from linear algebra and hyperbolic spaces that we need later. One can easily find a reference for these knowledge, such as Bridson and Haefliger's book \cite{BH}.
In Section \ref{equivalence1}, we define the notion of dominated representations on a subset of the group and Anosov representations over a closed, invariant subflow, then show the equivalence between them, i.e., (1)$\Leftrightarrow$(3) in Theorem \ref{MT1}.
In Section \ref{secsplitting}, we construct the dominated splitting of Anosov representations over closed, invariant subflows, then show that (1)$\Leftrightarrow$(2) in Theorem \ref{MT1}.
In Section \ref{sdpp}, we discuss the properties of the limit maps. In particular, we show that (3)$\Rightarrow$(4)$\Rightarrow$(2) in Theorem \ref{MT1} and (1) in Theorem \ref{MT2}.
In Section \ref{holdercts}, we prove that the limit maps of Anosov representations over closed, invariant subflows are Hölder continuous in a weaker sense, i.e., (2) in Theorem \ref{MT2}.
In Section \ref{secstability}, we show the stability of Anosov representations over closed, invariant subflows.
In Section \ref{theexamples}, we give several examples of non-discrete representations that are Anosov over subflows.

\subsection*{Acknowledgements}
The author would like to thank Nicolas Tholozan and Tengren Zhang for introducing him to the subject, many helpful discussions, and plenty of useful advice and comments.\\

\section{Preliminaries}\label{preliminaries}
In this section, we recall some basic knowledge of hyperbolic spaces, hyperbolic groups and their Cayley graphs, and also some properties of singular values of matrices and linear maps.\\

\subsection{$\delta$-hyperbolic spaces and hyperbolic groups}\ \\

We start by recalling some basic notions of hyperbolic spaces. See \cite[Part~III,Chapter~H]{BH} for reference.\\

Let $(X,d_1)$ and $(Y,d_2)$ be two geodesic metric spaces. We say a map $f:X\to Y$ is a \emph{$(\lambda, \epsilon)$-quasi-isometric embedding} for constants $\lambda\geqslant 1$  and $\epsilon\geqslant 0$, if $$\lambda^{-1}d_1(x_1,x_2)-\epsilon\leqslant d_2(f(x_1),f(x_2))\leqslant\lambda d_1(x_1,x_2)+\epsilon,$$ for any $x_1,x_2\in X$. We say $f$ is a \emph{$(\lambda,\epsilon)$-quasi-isometry} if it is moreover coarsely surjective, i.e., there exists $C\geqslant 0$, such that for any $y\in Y$, $d_2(y,f(x))\leqslant C$ for some $x\in X$.\\

We say a curve $\gamma: \mathbb{R} \rightarrow X$ is a \emph{$(\lambda,\epsilon)$-quasi-geodesic} for some constants $\lambda\geqslant1$ and $\epsilon\geqslant0$, if $\gamma$ is a $(\lambda,\epsilon)$-quasi-isometric embedding. We simply say  $\gamma$ is a quasi-geodesic if it is a $(\lambda,\epsilon)$-quasi-geodesic for some constants $\lambda\geqslant1$ and $\epsilon\geqslant0$. We say $\gamma$ is a \emph{geodesic} (with length parameter) if it is a $(1,0)$-quasi-geodesic. We often abuse the terminology and refer to the images of geodesics and quasi-geodesics as geodesics and quasi-geodesics respectively.\\

Let $\delta$ be a non-negative real number. A geodesic triangle in $X$ is said to be \emph{$\delta$-slim} if any side of it is contained in the closed $\delta$-neighborhood of the other two sides. We say $X$ is \emph{$\delta$-hyperbolic} if $X$ is a proper geodesic space and any geodesic triangle in $X$ is $\delta$-slim. We also say $X$ is \emph{hyperbolic} if $X$ is $\delta$-hyperbolic for some $\delta\geqslant 0$. Hyperbolicity is preserved by quasi-isometries.\\

Now we assume $X$ is a hyperbolic space. We say two geodesic rays $\gamma_1,\gamma_2:\mathbb{R}_+\to X$ are asymptotic if $t\mapsto d(\gamma_1(t),\gamma_2(t))$ is bounded. This defines an equivalence relation on the set of geodesic rays. The Gromov boundary of $X$ is the set $\partial X$ of all the equivalence classes of geodesic rays. For a geodesic $\gamma:\mathbb{R}\to X$, denote the equivalence class of $\gamma(\mathbb{R}_+)$ by $\gamma(+\infty)$ and the equivalence class of $\gamma(-\mathbb{R}_+)$ by $\gamma(-\infty)$. We call the pair $(\gamma(+\infty),\gamma(-\infty))$ the \emph{endpoints} of $\gamma$.\\

We fix a base point $p\in X$. The topology of $X\cup \partial X$ is given by the convergence described below. We say a curve $\gamma: \mathbb{R}_+\rightarrow X$ is a \emph{generalized geodesic ray} if there exists $R\geqslant 0$, such that $\gamma|_{[0,R]}$ is a geodesic segment and  $\gamma|_{[R,+\infty)}\equiv \gamma(R)$. In this case, we also denote $\gamma(+\infty)=\gamma(R)$. Let $\{x_n\}_{n\in \mathbb{N}}$ be a sequence of points in $X\cup \partial X$ and $x$ a point in $X\cup \partial X$.  Say $x_n$ converges to $x$ as $n\to +\infty$ if there exists a sequence of generalized geodesic rays $\{\gamma_n\}_{n\in \mathbb{N}}$, with $\gamma_n(0)=p$ and $\gamma_n(+\infty)=x_n$, such that any subsequence of $\{\gamma_n\}_{n\in \mathbb{N}}$ contains a convergent subsequence, uniformly when restricted to any compact set, to a generalized geodesic ray $\gamma$ with $\gamma(+\infty)=x$.
The topology defined is independent of the choice of the base point $p$. With this topology, $X\cup \partial X$ is compact and $\partial X$ is closed in $X\cup \partial X$.\\

One important property of hyperbolic spaces that we will use is the fellow traveler property, which is also called the Morse lemma in some references. A detailed proof can be find in \cite{BH} Part III, Theorem 1.7.

\begin{proposition}[Fellow Traveler Property]\label{FTP}
    Let $X$ be a $\delta$-hyperbolic space and let $c_1,c_2:\mathbb{R}\rightarrow X$ be two $(\lambda, \epsilon)$-quasi-geodesics in $X$ with the same endpoints. There exists a constant $R_0=R_0(\delta,\lambda,\epsilon)$ that depends only on $\delta,\lambda$ and $\epsilon$ such that the Hausdorff distance between $c_1$ and $c_2$ is less than $R_0$.
\end{proposition}

\begin{remark}
    The proposition is still true if we replace $c_1,c_2$ with $(\lambda, \epsilon)$-quasi-geodesic segments or $(\lambda, \epsilon)$-quasi-geodesic rays.
\end{remark}

If we extend the equivalence relation of being asymptotic to quasi-geodesic rays by saying that two quasi-geodesics are asymptotic if the Hausdorff distance between them is finite, then the set of equivalence classes of quasi-geodesics is canonically in bijection with $\partial X$.\\

As we will need more information about the parametrization of quasi-geodesics, we state the following stronger version of the fellow traveler property.

\begin{proposition}[Fellow Traveler Property with orientations]\label{FTPs} 
    Let $X$ be a $\delta$-hyperbolic space and $\lambda\geqslant 1$, $\epsilon \geqslant 0$ are two constants. Then there exists a constant $R=R(\delta,\lambda,\epsilon)$ that depends only on $\delta,\lambda$ and $\epsilon$ with the following property. Let $c_1$, $c_2:\mathbb{R}\rightarrow X$ be two arbitrary $(\lambda, \epsilon)$-quasi-geodesics in $X$ with $c_1(+\infty)=c_2(+\infty)$ and $c_1(-\infty)=c_2(-\infty)$. Let $R_0=R_0(\delta,\lambda,\epsilon)$ be the constant given by Proposition \ref{FTP}. For any $t_1 < t_2$ and $s_1\in\mathbb{R}$ such that $d(c_1(t_1),c_2(s_1))\leqslant R_0$, there exists $s_2> s_1$, such that $d(c_1(t_2),c_2(s_2))\leqslant R$.
\end{proposition}

\begin{proof}
    Let $R_1$ be a constant such that the Hausdorff distance between any two $(\lambda, \epsilon+R_0)$-quasi-geodesics in $X$ with the same endpoints is less then $R_1$. We take $R=2R_1+R_0$. Then the existence of $s_2$ follows from the fact that the Hausdorff distance between $c_1([t_1,+\infty])$ and $c_2([s_1,+\infty))$ is at most $2R_1+ d(c_1(t_1),c_2(s_1))\leqslant 2R_1+R_0=R$.\\
\end{proof}

We call the constant $R_0$ \emph{the upper bound of the Hausdorff distance between two $(\lambda, \epsilon)$-quasi-geodesics} and $R$ \emph{the order preserving upper bound of the Hausdorff distance between two $(\lambda, \epsilon)$-quasi-geodesics}. Intuitively, the proposition tells us that, if we pick two points $c_1(t_1)$ and $c_1(t_2)$ along $c_1$, we can pick two points $c_2(s_1)$ and $c_2(s_2)$ along $c_2$, such that the $c_1(t_i)$ and $c_2(s_i)$ ($i=1,2$) are close to each other and $s_2\geqslant s_1$ when $t_2\geqslant t_1$ (keep the order along the geodesics).\\

Let $\Gamma$ be a finitely generated group. For a finite, symmetric generating set $S$, i.e., $S=S^{-1}$ is finite, the \emph{Cayley graph} $\Cay(\Gamma,S)$ of $\Gamma$ with generating set $S$ is the metric graph (with edges of length 1) whose vertices are all the elements of $\Gamma$ and there is an edge between $\gamma$ and $\eta$ if and only if $\eta^{-1}\gamma\in S$. If $S$ and $S'$ are both finite, symmetric generating sets of $\Gamma$, then there is a quasi-isometry $\Cay(\Gamma, S)\rightarrow \Cay(\Gamma, S')$ that restricts to the $\id$ on $\Gamma$. We say $\Gamma$ is \emph{hyperbolic} if there exists a (hence, any) finite, symmetric generating set $S$ such that $\Cay(\Gamma, S)$ is hyperbolic. In the following sections, we always identify $\Gamma$ with the set of vertices in its Cayley graph.\\

\subsection{Singular values}\label{sectsv}\ \\

Let $(E,\langle\cdot,\cdot\rangle_1)$ and $(F,\langle\cdot,\cdot\rangle_2)$ be two real vector spaces equipped with inner products. The inner product $\langle\cdot,\cdot\rangle_1$ (respectively, $\langle\cdot,\cdot\rangle_2$) induces a identification of $E$ (respectively, $F$) and its dual space $E^*$ (respectively, $F^*$).\\

Let $\psi:E\to F$ be a linear map and let $\psi^*:F^*\to E^*$ be its dual. The map $\psi^*\psi :E\rightarrow E^*\simeq E$ is symmetric with respect to $\langle\cdot,\cdot\rangle_1$, and is diagonalizable with non-negative eigenvalues. The \emph{singular values} of $\psi$ are defined to be the eigenvalues of $(\psi^* \psi)^{\frac{1}{2}}$. Denote them by  $$\sigma_1(\psi)\geqslant \sigma_2(\psi)\geqslant\cdots\geqslant\sigma_d(\psi),$$ where $d=\dim E$. We say that $\psi$ has a \emph{gap of index $k$} if $\sigma_k(\psi)>\sigma_{k+1}(\psi)$.\\

Equivalently, we can define the singular values geometrically. The (operator) norm of $\psi$ is defined to be  $$\Vert \psi\Vert =\max_{v\in V,v\neq 0}\dfrac{\Vert \psi(v)\Vert _2}{\Vert v\Vert _1}$$ and the co-norm of $\psi$ is $$m(\psi)= \min_{v\in V,v\neq 0}\dfrac{\Vert \psi(v)\Vert _2}{\Vert v\Vert _1},$$ where $\Vert\cdot \Vert_1$ (respectively, $\Vert\cdot \Vert_2$) is the norm induced by $\langle \cdot,\cdot\rangle_1$ (respectively, $\langle \cdot,\cdot\rangle_2$). Then we have $$\Vert \psi\Vert =\sigma_1(\psi)\  \text{and}\ m(\psi) =\sigma_d(\psi).$$ Moreover, $$\sigma_k(\psi)=\max_{W\in \mathrm{Gr}_k(\mathbb{R}^d)} m(\psi|_W)=\min_{V\in \mathrm{Gr}_{d-k+1}(\mathbb{R}^d)} \Vert \psi|_V\Vert $$ for any $k\in\{1,2,...,d\}$.\\

For a matrix $A\in \mathrm{GL}(d,\mathbb{R})$, we can regard it as a linear transformation on $\mathbb{R}^d$  with the standard inner product and define the singular values of matrix $A$ to be the singular values of the linear transformation $A$. Observe that $\sigma_k(A)=\sigma_{d+1-k}(A^{-1})^{-1}$. Also, if $A$ has a gap of index $k$, i.e., $\sigma_k(A)>\sigma_{k+1}(A)$, we denote the eigenspace of $(AA^*)^{\frac{1}{2}}$ corresponding to the largest $k$ eigenvalues by $U_k(A)$ and denote the eigenspace of $(A^*A)^{\frac{1}{2}}$ corresponding to the smallest $d-k$ eigenvalues by $S_{d-k}(A)$. Then we have
$$ S_{d-k}(A)=U_{d-k}(A^{-1}),$$
$$ A\cdot S_{d-k}(A)= U_k(A)^{\perp},$$
$$ A^{-1}\cdot U_k(A)=S_{d-k}(A)^{\perp}.$$
One can verify that $U_k(A)\in \mathrm{Gr}_k(\mathbb{R}^d)$ is the unique $k$-plane such that $\sigma_k(A)=m(A|_{A^{-1}U_k(A)})$, and $S_{d-k}(A)\in \mathrm{Gr}_{d-k}(\mathbb{R}^d)$ is the unique $(d-k)$-plane such that $\sigma_{k+1}=\Vert A|_{S_{d-k}(A)}\Vert$.\\

The following observation is another useful estimate that follows from the geometric definition of singular values.

\begin{observation}\label{SingluarvalueEstimate}
    Let $A, B \in \mathrm{GL}(d,\mathbb{R})$ be two matrices, then $$\max\{m(A)\sigma_k(B),\sigma_k(A)m(B)\}\leqslant \sigma_k(AB)\leqslant\min\{\Vert A\Vert \sigma_k(B),\sigma_k(A)\Vert B\Vert \},$$ for any $1\leqslant k\leqslant d$.
\end{observation}\ \\ 

\subsection{Angles in Grassmannians}\ \\

Let $(\mathbb{R}^d,\langle\cdot,\cdot\rangle)$ be the real vector space of dimension $d$ equipped with an inner product and let $\Vert \cdot\Vert $ denote the norm defined by the inner product $\langle\cdot,\cdot\rangle$. Given two nonzero vectors $v,w\in \mathbb{R}^d$, the angle between $v,w$ is $\angle (v,w)=\arccos \dfrac{\langle v,w\rangle}{\Vert v\Vert \cdot\Vert w\Vert }$.\\

\begin{definition}
    Let $V,W\in \mathrm{Gr}_k(\mathbb{R}^d)$. We define the distance between $V$ and $W$ to be $$d(V,W)=\max \{\max_{v\in V,v\neq 0} \min_{w\in W,w\neq 0} \sin \angle(v,w),\max_{w\in W,w\neq 0} \min_{v\in V,v\neq 0} \sin \angle(v,w)\}$$
\end{definition}

Together with this distance, $\mathrm{Gr}_k(\mathbb{R}^d)$ is a compact metric space.\\

The proof of the following lemma can be found in \cite[Appendix~A.3]{BPS}.

\begin{lemma}\label{estsingularvalue}
    Let $A,B\in \mathrm{GL}(d,\mathbb{R})$ with $A$ and $AB$ having gaps of index $k$. Then 
    \begin{itemize}
        \item[(1).] For any unit vector $v\in\mathbb{R}^d$, we have
        $$ \Vert Av\Vert \geqslant \sigma_k(A) \sin \angle (v,S_{d-k}(A)),$$
        $$ \Vert A^{-1}v\Vert \geqslant \sigma_{k+1}(A)^{-1} \sin \angle (v,U_{k}(A));$$
        \item[(2).] For any $Q\in \mathrm{Gr}_{d-k}(\mathbb{R}^d)$, $P\in \mathrm{Gr}_k(\mathbb{R}^d)$, we have 
        $$ \Vert A|_Q\Vert \geqslant \sigma_k(A) d(Q,S_{d-k}(A)),$$
        $$ \Vert A^{-1}|_P\Vert \geqslant \sigma_{k+1}(A)^{-1} d(P,U_{k}(A));$$
        \item[(3).] $d(U_k(A),U_k(AB))\leqslant \Vert B\Vert \cdot\Vert B^{-1}\Vert \cdot \dfrac{\sigma_{k+1}(A)}{\sigma_k(A)};$
        \item[(4).] $d(B\cdot U_k(A),U_k(BA))\leqslant \Vert B\Vert \cdot\Vert B^{-1}\Vert \cdot \dfrac{\sigma_{k+1}(A)}{\sigma_k(A)} .$
    \end{itemize}
\end{lemma}\ \\

\section{Dominated representations on subsets and Anosov representations over closed subflows}\label{equivalence1}
In this section, we always assume that $\Gamma$ is a hyperbolic group with a fixed finite, symmetric generating set $S$. Let $\Cay(\Gamma,S)$ denote the Cayley graph of $\Gamma$. Let $\partial \Gamma$ denote the visual boundary of $\Gamma$ and $\partial^{(2)}\Gamma=\{ (x,y)|x,y\in \partial \Gamma,x\ne y\}$. We introduce the notions of dominated representations on subsets of $\Gamma$ and Anosov representations over closed, $\Gamma$-invariant subflows. Then we show the equivalence between them, i.e., the equivalence between (1) and (3) in Theorem \ref{MT1}. The idea of the proof is motivated by Bochi, Potrie and Sambarino's proof for \cite[Lemma~4.4]{BPS}, which deals with the classical Anosov representations.\\

\subsection{$k$-dominated representations on $\Gamma_P^+$}\label{kdom}\ \\

Let $P$ be a closed, $\Gamma$-invariant subset of $\partial^{(2)}\Gamma$. Let $G_P$ be the set of geodesics $l:\mathbb{R}\rightarrow \Cay(\Gamma,S)$ such that $l(0)=\id$ and $(l(+\infty),l(-\infty))\in P$ and let $$\Gamma_P^+=\{ l(n)\in \Gamma\ | \ l\in G_P\ \text{and}\ n\in\mathbb{Z}_+ \}.$$ Similarly, define $$\Gamma_P^-=\{ l(-n)\in \Gamma\ | \ l\in G_P\ \text{and}\ n\in\mathbb{Z}_+ \}.$$

\begin{definition}\label{defdomprop} Let $M$ be a subset of $\Gamma$.
    A representation $\rho: \Gamma\to \mathrm{GL}(d,\mathbb{R})$ is \emph{$k$-dominated on $M$} if there exist constants $C,\lambda>0$, such that for all $\gamma \in M$, $$\dfrac{\sigma_{k+1}(\rho(\gamma))}{\sigma_k(\rho(\gamma))}\leqslant Ce^{-\lambda|\gamma|},$$ where $|\gamma|$ is the distance between $\gamma$ and the identity in the Cayley graph, i.e., the word length of $\gamma$, and $\sigma_k(\rho(\gamma))$ is the $k$-th singular value of $\rho(\gamma)$.\\
\end{definition}

The definition of $\rho$ being $k$-dominated on $M$ does not depend on the choice of the generating set $S$.\\

We will need the following lemma in the proof later, which allows us to enlarge the set of $k$-dominated elements by providing weaker constants. The lemma directly comes from Observation \ref{SingluarvalueEstimate}.

\begin{lemma}\label{conjextending}
    Let $B$ be a finite subset of $\Gamma$. If $\rho$ is $k$-dominated on $\Gamma_{P}^+$, then there exist constants $C'$ and $\lambda'>0$ such that for any $\omega \in \Gamma_{P}^+$ and $\beta\in B$, we have $$\dfrac{\sigma_{k+1}(\rho(\beta\omega\beta^{-1}))}{\sigma_k(\rho(\beta\omega\beta^{-1}))} \leqslant C'e^{-\lambda'|\beta\omega\beta^{-1}|}.$$
\end{lemma}\ \\

\subsection{$k$-Anosov representations over $S_P$}\label{section3.2}\ \\

Firstly, we define the Gromov geodesic flow abstractly following the notations of Bochi--Potrie--Sambarino \cite{BPS}. Define $$\widetilde{U\Gamma}=\partial^{(2)}\Gamma\times \mathbb{R}=\{(x,y,s)|x,y\in \partial \Gamma,x\ne y,s \in\mathbb{R}\}$$ and define the lifted geodesic flow on $\widetilde{U\Gamma}$ by $\widetilde{\phi}^{t}(x,y,s)=(x,y,s+t)$. A \emph{cocycle} of the flow is a map $c:\Gamma\times\partial^{(2)}\Gamma \to \mathbb{R}$, such that $$c(\gamma\gamma',x,y)=c(\gamma,\gamma'x,\gamma'y)+c(\gamma',x,y).$$

By Gromov \cite{Gromov}, Mathéus \cite{Matheus}, Champetier \cite{Cham} and Mineyev \cite{Mineyev}, we have the following theorem.

\begin{theorem}[Gromov, Mathéus, Champetier, Mineyev]\label{metricunittangentbundle}
    There exists a metric $d$ on $\widetilde{U\Gamma}$ and a cocycle $c$, such that 
    \begin{itemize}
        \item[(1).] The action of $\Gamma$ on $\widetilde{U\Gamma}$ defined by $\gamma\cdot(x,y,s) =(\gamma x,\gamma y,s+c(x,y,\gamma))$ is properly discontinuous, cocompact and by isometries;
        \item[(2).] There exist constants $\lambda_1\geqslant 1$ and $\epsilon_1\geqslant 0$, such that each curve $\{ (x,y,t)\in \widetilde{U\Gamma}\ |\ t\in\mathbb{R} \}$ is a $(\lambda_1,\epsilon_1)$-quasi-geodesic for any $(x,y)\in\partial^{(2)}\Gamma$;
        \item[(3).] The $\mathbb{Z}/2\mathbb{Z}$ action on $\widetilde{U\Gamma}$ defined by $(x,y,t)\rightarrow (y,x,-t)$ for any $(x,y)\in\partial^{(2)}\Gamma$ commutes with the $\Gamma$ action;
        \item[(4).] $c$ is positive, i.e., $c(\gamma^+,\gamma^-,\gamma)>0$ for any $\gamma\in\Gamma$ of infinite order, where $\gamma^+$ and $\gamma^-$ denote the attractor and repeller of $\gamma$ respectively;
        \item[(5).] $\phi^t$ is bi-Lipschitz.
\end{itemize}
\end{theorem}

Fix a base point $p\in \widetilde{U\Gamma}$. Then by Theorem 60(c) in \cite{Mineyev}, there exist constants $\lambda_2\geqslant 1$ and $\epsilon_2\geqslant 0$, such that the orbit map $\tau_p: \Cay(\Gamma,S)\rightarrow \widetilde{U\Gamma}$, defined by $\tau_p(\gamma)=\gamma p$ for all $\gamma\in\Gamma$, then extending linearly to the edges of $\Cay(\Gamma,S)$, is a $(\lambda_2,\epsilon_2)$-quasi-isometry. In particular, $\widetilde{U\Gamma}$ is hyperbolic, there is a homeomorphism $\partial\Gamma\simeq \partial \widetilde{U\Gamma}$, and with this identification, $(x,y,+\infty)=x$ and $(x,y,-\infty)=y$ by Theorem \ref{metricunittangentbundle} (4).\\

We denote the $\mathbb{Z}/2\mathbb{Z}$-action on $\widetilde{U\Gamma}$ in Theorem \ref{metricunittangentbundle} (3) by $z\mapsto \hat{z}$. Let $U\Gamma=\widetilde{U\Gamma}/\Gamma$ be the quotient and let $\phi$ denote the geodesic flow induced by $\widetilde{\phi}$ on $U\Gamma$. Then the $\mathbb{Z}/2\mathbb{Z}$-action descends to $U\Gamma$.\\

Intuitively, we may regard $\partial^{(2)}\Gamma$ as the collection of geodesics of $\Cay(\Gamma,S)$ and $U\Gamma$ as the unit tangent bundle of $\Cay(\Gamma,S)/\Gamma$. These are not exact since the geodesic between two distinct points in $\partial \Gamma$ may not be unique. However, in the sense of coarse geometry, we can still keep these intuitions in mind due to the Fellow Traveler Property (see Proposition \ref{FTP} and Proposition \ref{FTPs}).\\

Now, we give the definition of Anosov representations over certain subflows.\\

Let $\widetilde{E}=\widetilde{U\Gamma}\times\mathbb{R}^d$ be the trivial $\mathbb{R}$-vector bundle of rank $d$ over $\widetilde{U\Gamma}$ and define the flow on the vector bundle $\widetilde{E}$  to be the automorphism $\widetilde{\psi}^t$ such that $\widetilde{\psi}^t((x,y,s),v)=(\widetilde{\phi}^t(x,y,s),v)$ for any $(x,t,s)\in \widetilde{U\Gamma}$ and $v\in \mathbb{R}^d$. For a representation $\rho: \Gamma\rightarrow \mathrm{GL}(d,\mathbb{R})$, the $\Gamma$ action on $\widetilde{E}$ induced by $\rho$ is defined to be $\gamma\cdot(x,y,s,v)=(\gamma\cdot(x,y,s),\rho(\gamma)v).$ We denote $E_\rho= \widetilde{E} / \Gamma$. As $\widetilde{\psi}$ commutes with the $\Gamma$ action, we denote $\psi$ the flow on $E_\rho$ induced by $\widetilde{\psi}$.\\

Let $P$ be a closed, $\Gamma$-invariant subset of $\partial^{(2)}\Gamma$. Let $\widetilde{S_P}= P \times \mathbb{R} \subset \widetilde{U\Gamma}$. Notice that $\widetilde{S_P}$ is $\Gamma$-invariant and $\widetilde{\phi}$-invariant. We also denote $\widetilde{E}|_{\widetilde{S_P}}=\widetilde{S_P}\times\mathbb{R}^d$, $S_P=\widetilde{S_P}/ \Gamma$ and $E_\rho|_{S_P}=\widetilde{E}|_{\widetilde{S_P}}/ \Gamma$.\\

Here are two (Riemannian) metrics on the vector bundle $\widetilde{E}$ (i.e., continuous families of inner products on the fibers of $\widetilde{E}$,) we will consider.

\begin{notation}\label{twonorm}\ \\
\begin{itemize}
    \item[(1).] Endow each fiber of the trivial bundle $\widetilde{E}$ with the standard inner product on $\mathbb{R}^d$. Denote the norm induced by this metric by $\Vert \cdot\Vert_0$. Observe that $\Vert \cdot\Vert_0$ is $\widetilde{\psi}^t$-invariant, i.e. $\Vert v \Vert_{0,z} = \Vert \psi^t(v) \Vert_{0,\phi^t(z)}=\Vert v \Vert_{0,\phi^t(z)}$ for any $v\in \mathbb{R}^d$, $z\in \widetilde{U\Gamma}$ and $t\in\mathbb{R}$. Hence, for convenience, we omit the subscripts of points for $\Vert \cdot\Vert_0$. The metric will be used to measure the singular values of matrix $\rho(\gamma)$ for group elements $\gamma$.\\
    \item[(2).] Fix a $\mathbb{Z}/2\mathbb{Z}$-invariant Riemannian metric on $E_\rho$. Then it lifts to a $\Gamma$-invariant and also $\mathbb{Z}/2\mathbb{Z}$-invariant metric on $\widetilde{E}$, where the norm is denoted by $\Vert \cdot\Vert$. Since it comes from $E_\rho$, we use it to define the singular values of the flow ${\psi}$.\\
\end{itemize}
\end{notation}

\begin{definition}\label{DefAnosovoverSp}
    A representation $\rho: \Gamma\rightarrow \mathrm{GL}(d,\mathbb{R})$ is \emph{$k$-Anosov over ${S_P}$} if there exist constants $C,\lambda>0$, such that for any $q \in{S_P}$ and $t\in \mathbb{R}_+$, the linear map $\psi^t_q:E_{\rho,q}\to E_{\rho,\phi(q)}$ satisfies $$\dfrac{\sigma_{d-k+1}(\psi^t_q)}{\sigma_{d-k}(\psi^t_q)}\leqslant Ce^{-\lambda t}.$$
\end{definition}

Since $S_P$ is compact, this definition does not depends on the choice of $\Vert \cdot\Vert$.

\begin{remark}
    Let $\hat{P}=\{(y,x) \ |\ (x,y)\in P\}$. Notice that $(\psi^t_q)^{-1} = \psi^{-t}_{\phi^t(q)}=\psi^t_{\phi^{-t}(\hat{q})}$ if we identify $E_{\rho,q}$ with $E_{\rho,\hat{q}}$ for any $q\in \widetilde{U\Gamma}$ and $t\in\mathbb{R}$. Hence, $\rho$ is $k$-Anosov over $S_P$ if and only if it is $(d-k)$-Anosov over $S_{\hat{P}}$. In general, $\rho$ being $(d-k)$-Anosov over $S_P$ and $\rho$ being $k$-Anosov over $S_P$ are not equivalent. See Section \ref{directedanosov} for example. When $P=\partial^{(2)}\Gamma$, we say $\rho$ is $k$-Anosov, which coincides with the notion of classical Anosov representations.
\end{remark}\ \\

\subsection{The equivalence}\label{equivalence2}\ \\

Now we are ready to prove the equivalence between (1) and (3) in Theorem \ref{MT1}.

\begin{theorem}\label{mainthm1}
    Let $P$ be a closed, $\Gamma$-invariant subset of $\partial^{(2)}\Gamma$. Then $\rho$ is $k$-Anosov over $S_P$ if and only if $\rho$ is $k$-dominated on $\Gamma_P^+$.
\end{theorem}

We fix a compact set $K_0\subset \widetilde{U\Gamma}$ such that $\Gamma\cdot K_0=\widetilde{U\Gamma}$ and a base point $p\in K_0$. By Theorem \ref{metricunittangentbundle}, there are constants $\lambda_1,\lambda_2\geqslant 1$ and $\epsilon_1,\epsilon_2\geqslant 0$, such that $\{ (x,y,t)\in \widetilde{U\Gamma}\ |\ t\in\mathbb{R} \}$ is a $(\lambda_1,\epsilon_1)$-quasi-geodesic for any $(x,y)\in\partial^{(2)}\Gamma$ and the orbit map $\tau_p: \Cay(\Gamma,S)\rightarrow \widetilde{U\Gamma}$ is a $(\lambda_2,\epsilon_2)$-quasi-isometry.\\

To prove the theorem, we use the following pair of lemmas. The first one describes the relation of $t\in\mathbb{R}$ and the word length of $\gamma\in\Gamma$ when $\widetilde{\phi}^t(z)$ and $\gamma z$ are close for some $z\in \widetilde{U\Gamma}$.

\begin{lemma}\label{disrelation}
    For any compact set $K\subset \widetilde{U\Gamma}$ with $p\in K$, there exist constants $a,b>0$ such that for any $z\in K$, if $t\in \mathbb{R}$ and $\gamma\in \Gamma$ satisfy $\gamma^{-1}\widetilde{\phi}^t(z)\in K$, then $a^{-1}|t|-b\leqslant|\gamma|\leqslant a|t|+b$. 
\end{lemma}
\begin{proof}
    We denote $\diam(K)=sup\{d(z_1,z_2),\ z_1,z_2\in K\}$. Since $d(\gamma z, \widetilde{\phi}^t(z))=d(z, \gamma^{-1}\widetilde{\phi}(z)) \leqslant \diam(K)$ and $d(z,p)=d(\gamma z,\gamma p)\leqslant \diam(K)$, we have $$ |d(p,\gamma p) - d(z,\widetilde{\phi}^t(z))|\leqslant d(z,p)+d(\gamma z,\gamma p)+ d(\gamma z, \widetilde{\phi}^t(z))\leqslant 3\diam(K).$$ By Theorem \ref{metricunittangentbundle} (1), (3) and the Milnor-Svarc Lemma, $\lambda_1^{-1}|t|-\epsilon_1\leqslant d(z,\phi^t(z))\leqslant \lambda_1|t|+\epsilon_1$ and $\lambda_2^{-1}|\gamma|-\epsilon_2 \leqslant d(p, \gamma p) \leqslant  \lambda_2|\gamma|+\epsilon_2$. By $$-3\diam(K)\leqslant d(p,\gamma p) - d(z,\widetilde{\phi}^t(z)) \leqslant 3\diam(K),$$ we have $$-3\diam(K)\leqslant  \lambda_2|\gamma|+\epsilon_2 - \lambda_1^{-1}|t|+\epsilon_1 $$ and $$\lambda_2^{-1}|\gamma|-\epsilon_2 - \lambda_1|t|-\epsilon_1 \leqslant 3\diam(K).$$ Therefore, $$(\lambda_1\lambda_2)^{-1}|t|-\lambda_2^{-1}(\epsilon_1+\epsilon_2+3\diam(K))\leqslant |\gamma|\leqslant (\lambda_1\lambda_2)|t|+\lambda_2(\epsilon_1+\epsilon_2+3\diam(K)).$$
\end{proof}

The second lemma tells that for any $z\in K_0\cap\widetilde{S_P}$ and $t\in\mathbb{R}_+$, we can find $\gamma\in \Gamma _P^+$, such that $\gamma z$ is close to $\widetilde{\phi}^t(z)$, and vice versa.

\begin{lemma}\label{equiesti}
    There exists a compact set $K\subset \widetilde{U\Gamma}$ containing $K_0$ and a finite subset $B$ of $\Gamma$, such that
    \begin{itemize}
        \item[(1).] for any $z\in K_0\cap\widetilde{S_P}$ and $t\in\mathbb{R}_+$, there exists $\gamma\in \{\beta\omega\beta^{-1}\ |\ \beta\in B, \omega\in\Gamma_P^+\}$, such that $\gamma^{-1}\widetilde{\phi}^t(z)\in K$;
        \item[(2).] for any $\gamma'\in \Gamma_P^+$, there exists $z'\in K\cap\widetilde{S_P}$ and $t'\in\mathbb{R}_+$, such that $\gamma'^{-1}\widetilde{\phi}^{t'}(z')\in K$.
    \end{itemize}
\end{lemma}

\begin{proof}
    We know that $\widetilde{U\Gamma}$ is $\delta$-hyperbolic for some constant $\delta\geqslant 0$.
    Let $C=\max(\lambda_1,\lambda_2)$ and $\epsilon=\max(\epsilon_1,\epsilon_2)$.
    By Proposition \ref{FTP} and Proposition \ref{FTPs}, there exist constants
    $R(\delta,C,\epsilon)\geqslant R_0\geqslant 0$, where $R_0$ is the upper bound of the Hausdorff distance between two $(C, \epsilon)$-quasi-geodesics
    and $R$ is the order preserving upper bound of the Hausdorff distance between two $(C, \epsilon)$-quasi-geodesics.
    Let $K$ be the closed ball centered at $p$ of radius $R_1=R_0+R+2\lambda_2+2\epsilon_2+\diam(K_0)$.\\
\\
    (1).  We write $z=(x,y,s)\in K_0$ with $(x,y)\in P$ and $s\in \mathbb{R}$. Denote the $(\lambda_1,\epsilon_1)$-quasi-geodesic $\{(x,y,t)|t\in\mathbb{R}\}$ by $l_1$
    and denote the $\tau_p$ image of a geodesic $L_2$ in $\Cay(\Gamma,S)$ with endpoints $(x,y)$ by $l_2$, which is a $(\lambda_2,\epsilon_2)$-quasi-geodesic. 
    The Hausdorff distance between $l_1$ and $l_2$ is less than $R_0$, hence
    we can find $\beta\in\Gamma$ and $\beta\omega\in\Gamma$ on the geodesic $L_2$, such that $d(\beta p,z)\leqslant R_0+\lambda_2+\epsilon_2$,
    $d(\beta\omega p,\widetilde{\phi}^t(z))\leqslant R+\lambda_2+\epsilon_2$ and $\beta^{-1}\beta\omega=\omega\in\Gamma^+_P$.\\

    Let $\gamma =\beta\omega\beta^{-1}$, then
    $$d(p, \gamma^{-1}\widetilde{\phi}^t({z}))=d(\gamma p,\widetilde{\phi}^t({z}))\leqslant 
    d(\gamma p,\gamma \beta p)+ d(\gamma \beta p,\widetilde{\phi}^t({z})) = d(p,\beta p)+d(\beta\omega p,\widetilde{\phi}^t({z}))\leqslant $$ 
    $$ d(\beta p,z)+d(z,p)+d(\beta\omega p,\widetilde{\phi}^t({z})) \leqslant R_0+\lambda_2+\epsilon_2+\diam(K_0)+R+\lambda_2+\epsilon_2=R_1.$$
    Let $B=\{ \beta\in \Gamma\ |\ |\beta|\leqslant \lambda_2R_1+\epsilon_2\}\supset \{ \beta\in \Gamma\ |\ \beta p \in K \}$.
    Then (1) is proved.\\
\\
    (2). Since $\gamma'\in \Gamma_P^+$, let $L_2'$ be a geodesic in $G_P$ with $L_2'(n)=\gamma'$ for an positive integer $n$. Then $\tau_p(L_2')$ passes through $p$ and $\gamma p$.
    We denote $(L_2'(+\infty),L_2'(-\infty))=(x',y')$. Let  $l_1'$ be the $(\lambda_1,\epsilon_1)$-quasi-geodesic $\{(x',y',t)|t\in\mathbb{R}\}$.
    The Hausdorff distance between $l_1'$ and $\tau_p(L_2')$ is less than $R_0$.
    Then there exists $s'\in\mathbb{R}$ and $t'\in\mathbb{R}_+ $ such that the point $z'=(x',y',s')$ satisfies $d(z',p)\leqslant R_0$ and 
    $d(\widetilde{\phi}^{t'}(z'), \gamma' p)=d(p,\gamma'^{-1}\widetilde{\phi}^{t'}(z'))\leqslant R$. Hence $z'\in K$ and $\gamma'^{-1}\widetilde{\phi}^{t'}(z')\in K$.\\
\end{proof}

Now we start the proof of the theorem.

\begin{proof}[Proof of Theorem \ref{mainthm1}]
    Let $K,B$ be as given in Lemma \ref{equiesti}, and let $a,b$ be the constants given by applying Lemma \ref{disrelation} to $K$. Since $K$ is compact, the two norms are bi-Lipschitz on $K$, that is, there exists a constant $C_0>0$ such that $$C_0^{-1}\Vert \cdot\Vert\leqslant\Vert \cdot\Vert_0
    \leqslant C_0\Vert\cdot\Vert$$ on $K$.\\

    Firstly, we assume $\rho$ is $k$-dominated on $\Gamma_P^+$. Fix an arbitrary point $q\in S_P$ and $t\in\mathbb{R}_+ $, there is a lift $z$ of $q$ in $K_0$. By Lemma \ref{equiesti} (1), there exists $\gamma\in \{\beta\omega\beta^{-1}\ |\ \beta\in B, \omega\in\Gamma_P^+\}$, such that $\gamma^{-1}\widetilde{\phi}^t(z)\in K$.\\

    Let $v\in \mathbb{R}^d$ be an arbitrary nonzero vector.
    We may consider $v$ as a vector in any fiber of $\widetilde{E}$.
    Recall Notation \ref{twonorm}.
    Since $\Vert \cdot\Vert_0$ and $\Vert \cdot\Vert$ are bi-Lipschitz on $K$ and the flow $\widetilde{\psi^t}$ is trivial on the trivial bundle $\widetilde{E}$, we have 
    $$ C_0^{-2}\dfrac{\Vert \rho(\gamma^{-1})v\Vert_0}{\Vert v\Vert_0}
    =
    C_0^{-2}\dfrac{\Vert \rho(\gamma^{-1})\widetilde{\psi}^t_z(v)\Vert_0}{\Vert v\Vert_0}
    \leqslant
    \dfrac{\Vert \rho(\gamma^{-1})\widetilde{\psi}^t_z(v)\Vert_{\gamma^{-1}\widetilde{\phi}^t(z)}}{\Vert v\Vert_{z}}$$
    $$\leqslant C_0^2\dfrac{\Vert \rho(\gamma^{-1})\widetilde{\psi}^t_z(v)\Vert_0}{\Vert v\Vert_0}
    = C_0^2\dfrac{\Vert \rho(\gamma^{-1})v\Vert_0}{\Vert v\Vert_0},$$
    and since $\Vert \cdot\Vert$ is $\Gamma$-invariant, we have
    $$\dfrac{\Vert \widetilde{\psi}^t_z(v)\Vert_{\widetilde{\phi}^t(z)}}{\Vert v\Vert_{z}}
    =\dfrac{\Vert \rho(\gamma^{-1})\widetilde{\psi}^t_z(v)\Vert_{\gamma^{-1}\widetilde{\phi}^t(z)}}{\Vert v\Vert_{z}}.$$
    
    Then by the geometric definition of singular values, there exists $C_1>0$, such that $C_1^{-1}\sigma_j(\rho(\gamma^{-1}))
    \leqslant \sigma_j(\psi^t_z)\leqslant C_1\sigma_j(\rho(\gamma^{-1}))$ for any $1\leqslant j\leqslant d$.\\

    Since $\sigma_{d-k+1}(\rho(\gamma^{-1}))=\sigma_{k}(\rho(\gamma))^{-1}$ and 
    $\sigma_{d-k}(\rho(\gamma^{-1}))=\sigma_{k+1}(\rho(\gamma))^{-1}$, we have
    $$\dfrac{\sigma_{d-k+1}(\widetilde{\psi}^t_z)}{\sigma_{d-k}(\widetilde{\psi}^t_z)}\leqslant C_1^2
    \dfrac{\sigma_{d-k+1}(\rho(\gamma^{-1}))}{\sigma_{d-k}(\rho(\gamma^{-1}))} = C_1^2 \dfrac{\sigma_{k+1}(\rho(\gamma))}{\sigma_{k}(\rho(\gamma))}.$$

    By Lemma \ref{disrelation},  there exist constants $a,b>0$, such that $a^{-1}|t|-b\leqslant|\gamma|\leqslant a|t|+b$.
    By Lemma \ref{conjextending}, there exist constants $C_2$ and $c_2>0$ such that
    $$\dfrac{\sigma_{k+1}(\rho(\gamma))}{\sigma_k(\rho(\gamma))} \leqslant C_2e^{-c_2|\gamma|}.$$
    Then $$\dfrac{\sigma_{d-k+1}(\psi^t_q)}{\sigma_{d-k}(\psi^t_q)}=\dfrac{\sigma_{d-k+1}(\widetilde{\psi}^t_z)}{\sigma_{d-k}(\widetilde{\psi}^t_z)}
    \leqslant C_1^2C_2e^{-c_2|\gamma|} \leqslant C_1^2C_2e^{-c_2(a^{-1}t-b)},$$
    for any $q\in S_P$ and $t\in\mathbb{R}_+ $. Therefore, $\rho$ is $k$-Anosov over $S_P$.\\
    
    On the other hand, we assume $\rho$ is $k$-Anosov over $S_P$ and show that $\rho$ is $k$-dominated on $\Gamma_P^+$. Fix an arbitrary $\gamma'\in \Gamma_P^+$.
    By Lemma \ref{equiesti} (2), there exists $z'\in K\cap\widetilde{S_P}$ and $t'>0$, such that $\gamma'^{-1}\widetilde{\phi}^{t'}(z')\in K$. Let $q'\in S_P$
    be the projection of $z'$ through the quotient $\widetilde{S_P}\rightarrow \widetilde{S_P}/\Gamma=S_P$.
    By the same estimate as above
    we have $$\dfrac{\sigma_{k+1}(\rho(\gamma'))}{\sigma_{k}(\rho(\gamma'))}
    \leqslant C_1^2 \dfrac{\sigma_{d-k+1}(\widetilde{\psi}^{t'}_{z'})}{\sigma_{d-k}(\widetilde{\psi}^{t'}_{z'})}
    = C_1^2 \dfrac{\sigma_{d-k+1}({\psi}^{t'}_{q'})}{\sigma_{d-k}({\psi}^{t'}_{q'})}.$$
    Since $\rho$ is $k$-Anosov over $S_P$, 
    there exist constants $C_3,c_3>0$, such that $$\dfrac{\sigma_{d-k+1}(\psi^{t'}_{z'})}{\sigma_{d-k}(\psi^{t'}_{z'})}\leqslant C_3e^{-c_3 t'},$$
    for any $z' \in{S_P}$ and $t'\in \mathbb{R}_+$.
    Then we apply Lemma \ref{disrelation} again,
    we have $$\dfrac{\sigma_{k+1}(\rho(\gamma'))}{\sigma_{k}(\rho(\gamma'))}\leqslant C_1^2C_3e^{-c_3 a^{-1}(|\gamma'|-b)},$$
    for any $\gamma'\in \Gamma_P^+$. Therefore, $\rho$ is $k$-dominated on $\Gamma_P^+$.\\
\end{proof}\ \\
\section{Dominated splitting of Anosov representations over closed subflows}\label{secsplitting}
In this section, we show the equivalence of (1) and (2) in Theorem \ref{MT1}.\\

We still assume $\Gamma$ is a hyperbolic group with a finite, symmetric generating set $S$ and
$P$ is a closed, $\Gamma$-invariant subset of $\partial^{(2)}\Gamma$.
Fix a representation $\rho:\Gamma \rightarrow \mathrm{GL}(d,\mathbb{R})$. To simplify the notation, we will abbreviate $E=E_\rho$. 

\begin{definition}\label{defdomsplit}
    We say $E$ admits a \emph{$k$-dominated splitting over $S_P$}, if $E|_{S_P}=E^s\oplus E^u$, where $E^s$ and $E^u$ are $\psi$-invariant subbundles of $E|_{S_P}$ over $S_P$, with $\mathrm{rank}(E^s)=k$, such that there exist constants $C,\lambda>0$, $$\frac{\Vert \psi_q^t(v)\Vert }{\Vert \psi_q^t(w)\Vert }\leqslant Ce^{-\lambda t},$$ for any $q\in S_P, t\in \mathbb{R}_{+}$ and $v\in E^s_{q}$, $w\in E^u_{q}$, with $\Vert v\Vert =\Vert w\Vert =1$. We call $E^s$ the stable direction and $E^u$ the unstable direction of $E$.
\end{definition}

The definition is independent of the choice of the Riemannian metric on $E$ since $S_P$ is compact.

\begin{remark}\label{midsingular}
    It is easy to see that if $E|_{S_P}=E^s\oplus E^u$ is a $k$-dominated splitting of $E$ over $S_P$, then there exist constants $C,\lambda>0$, such that
    $$\frac{\Vert \psi_q^t|_{E^s_{q}}\Vert }{m(\psi_q^t|_{E^u_{q}})}\leqslant Ce^{-\lambda t},$$ for any $q\in S_P$ and $t\in \mathbb{R}_{+}$, where $\Vert \cdot\Vert $ denotes the
    largest singular value i.e., the norm, and $m(\cdot)$ denote the smallest singular value, i.e., the co-norm with respect to the Riemannian metric.
\end{remark}

\begin{theorem}\label{equivdands}
    $\rho$ is $k$-Anosov over $S_P$ if and only if $E$ admits a $k$-dominated splitting over $S_P$.
\end{theorem}

One proof of the theorem can be given by a classical result of Bochi--Gourmelon \cite{BG}, which based on the properties of the flow $\psi$. Later in Section \ref{sdpp}, we will introduce another algebraic approach based on the $k$-dominated property of $\rho$. It is easy to see that once $E$ admits a $k$-dominated splitting over $S_P$, it is unique.\\

We recall Bochi and Gourmelon's theorem.

\begin{theorem}[\cite{BG} Theorem A]\label{BGtheorem}
    Let $V$ be a real vector bundle of rank $k$ over a compact Hausdorff space $X$. Let $A:V\rightarrow V$ be a automorphism of vector bundle $V$ fibring over a homeomorphism $T:X\rightarrow X$. Then the following two conditions are equivalent.\\
    \begin{itemize}
        \item [(1).] $V$ admits a $k$-dominated splitting, that is, there exist two $A$-invariant subbundles $W$ and $U$ of $V$, constants $C,\lambda>0$,  such that $V=W\oplus U$, $\mathrm{rank}(W)=k$ and $$\frac{\Vert A^n(x)|_{W_x}\Vert}{m(A^n(x)|_{U_x})}\leqslant Ce^{-\lambda n },$$ for any $x\in X$ and $n\in\mathbb{N}$. Moreover, $$W_x= \lim\limits_{n\rightarrow +\infty} S_k(A^n(x)),$$ and $$U_x= \lim\limits_{n\rightarrow +\infty} U_{d-k}(A^n(T^{-n}(x))),$$ for any $x\in X$.\\
        \item [(2).] There exist constants $C,\lambda>0$, such that $$\frac{\sigma_{d-k+1}(A^n(x))}{\sigma_{d-k}(A^n(x))}\leqslant Ce^{-\lambda n },$$ for any $x\in X$.
    \end{itemize}
\end{theorem}

It remains to pass from discrete to continuous. Recall the notations $S_k(\cdot)$ and $U_k(\cdot)$ defined in section \ref{sectsv}.

\begin{lemma}\label{flowcauchy}
    If $\rho$ is $k$-Anosov over $S_P$, then the limit of $S_k(\psi^t_q)$ as $t\rightarrow +\infty$ exists and the convergence is uniform for all $q\in S_P$.
\end{lemma}
\begin{proof}
    Let $q\in S_P$ be an arbitrary point and $t_2>t_1>0$. Let $v\in S_k(\psi^{t_1}_q)$ be a unit vector such that $sin\angle (v, S_k(\psi^{t_2}_q) )=d(S_k(\psi^{t_1}_q), S_k(\psi^{t_2}_q))$  Decompose $v=v'+v''$ with $v'\in S_k(\psi^{t_2}_q)^{\perp}$ and $v''\in S_k(\psi^{t_2}_q)$ and observe that $\Vert v'\Vert =d(S_k(\psi^{t_1}_q), S_k(\psi^{t_2}_q)).$\\

    Since $(\psi^{t_2}_q)^*\psi^{t_2}_q$ preserves $S_k(\psi^{t_2}_q)$, we have $0=\langle v',(\psi^{t_2}_q)^*\psi^{t_2}_q (v'') \rangle=\langle\psi^{t_2}_q(v'),\psi^{t_2}_q (v'')\rangle,$ Then $$\Vert \psi^{t_2}_q(v)\Vert =\Vert \psi^{t_2}_q(v')+\psi^{t_2}_q(v'')\Vert \geqslant \Vert \psi^{t_2}_q(v')\Vert \geqslant \sigma_{d-k}(\psi^{t_2}_q)\Vert v'\Vert.$$
    
    On the other hand, $\Vert \psi^{t_2}_q(v)\Vert =\Vert \psi^{t_2-t_1}_{\phi^{t_1}(q)}\psi^{t_1}_q(v)\Vert \leqslant \Vert \psi^{t_2-t_1}_{\phi^{t_1}(q)}\Vert \cdot\Vert \psi^{t_1}_q(v)\Vert  \leqslant \Vert \psi^{t_2-t_1}_{\phi^{t_1}(q)}\Vert \cdot \sigma_{d-k+1}(\psi^{t_1}_q),$ as $v\in S_k(\psi^{t_1}_q)$. Then together with Observation \ref{SingluarvalueEstimate}, $$d(S_k(\psi^{t_1}_q), S_k(\psi^{t_2}_q)) =\Vert v'\Vert \ \leqslant\ \frac{\sigma_{d-k+1}(\psi^{t_1}_q)}{\sigma_{d-k}(\psi^{t_2}_q)}\cdot \Vert \psi^{t_2-t_1}_{\phi^{t_1}(q)}\Vert
    \ \leqslant\ \frac{\sigma_{d-k+1}(\psi^{t_1}_q)}{\sigma_{d-k}(\psi^{t_1}_q)}\cdot \frac{\Vert \psi^{t_2-t_1}_{\phi^{t_1}(q)}\Vert}{m( \psi^{t_2-t_1}_{\phi^{t_1}(q)})}.$$
    
    We firstly consider the case when $t_2-t_1\leqslant 1$. Since $[0,1]$ and $S_P$ are compact, there exists $C'>0$, such that $\Vert \psi^{t}_{q}\Vert\leqslant C'$ and  $\Vert \psi^{-t}_{q}\Vert \leqslant C'$, for any $t\in [0,1]$ and $q\in S_P$.\\
    
    Since $\rho$ is $k$-Anosov over $S_P$, there exist constants $C,\lambda>0$, such that $\dfrac{\sigma_{d-k+1}(\psi^t_q)}{\sigma_{d-k}(\psi^t_q)}\leqslant Ce^{-\lambda t}$ for any $q\in S_P$ and $t\in\mathbb{R}_+$. Therefore, $$d(S_k(\psi^{t_1}_q), S_k(\psi^{t_2}_q))\ \leqslant \ C C'^2\cdot e^{-\lambda t_1}.$$
    
    For arbitrary $t_2>t_1>0$, we separate the interval $[t_1,t_2]$ into segments of length $1$ and apply the above discussion. We have that $$d(S_k(\psi^{t_1}_q), S_k(\psi^{t_2}_q))\ \leqslant \ C C'^2\cdot (e^{-\lambda t_1} + e^{-\lambda (t_1+1)} + e^{-\lambda (t_1+2)} + ...)\leqslant 2CC'^2\dfrac{1}{1-e^{-\lambda}}\cdot e^{-\lambda t_1}.$$
    
    Hence $S_k(\psi^{t}_q)$ is uniformly Cauchy for any $q\in S_P$ as $t\rightarrow +\infty$.\\
\end{proof}

\begin{proof}[Proof of Theorem \ref{equivdands}]
    It is not hard to see that if $E$ admits a $k$-dominated splitting over $S_P$, then $\rho$ is $k$-Anosov over $S_P$ by Remark \ref{midsingular}.\\

    Now we assume $\rho$ is $k$-Anosov over $S_P$. We apply Theorem \ref{BGtheorem} for $V=E|_{S_P}$, $X= S_P$, $A=\psi^1$ and $T=\phi^1$. Then we have $E|_{S_P}= E^s\oplus E^u$ with $E^s$, $E^u$ two $\psi^1$-invariant subbundles of $E|_{S_P}$ and there exist constants $C_0,\lambda>0$, such that  $$\frac{\Vert \psi^n_q|_{E^s_q}\Vert}{m(\psi^n_q|_{E^u_q})}\leqslant C_0e^{-\lambda n },$$ for any $q\in S_P$ and $n\in\mathbb{N}$.  Then by Lemma \ref{flowcauchy}, $\lim\limits_{t\rightarrow +\infty}S_k(\psi^t_q) = \lim\limits_{n\rightarrow +\infty} S_k(\psi^n_q) = E^s$ and similarly,  $\lim\limits_{t\rightarrow +\infty} U_{d-k}(\psi^t_{\phi^{-t}(q)}) = \lim\limits_{n\rightarrow +\infty} U_{d-k}(\psi^n_{\phi^{-n}(q)}) = E^u$. Therefore, $E^s$, $E^u$ are $\psi^t$-invariant subbundles.\\

    Since $[0,1]$ and $S_P$ are compact, there exists a constant $C_1>0$, such that $$C_1^{-1} \leqslant \inf_{t\in [0,1],q\in S_P} \frac{\Vert \psi_q^t|_{E^s_{q}}\Vert }{m(\psi_q^t|_{E^u_{q}})} \leqslant  \sup_{t\in [0,1], q\in S_P} \frac{\Vert \psi_q^t|_{E^s_{q}}\Vert }{m(\psi_q^t|_{E^u_{q}})} \leqslant C_1.$$ Then $$\frac{\Vert \psi^t_q|_{E^s_q}\Vert}{m(\psi^t_q|_{E^u_q})}\leqslant C_1C_0e^{-\lambda (t-1) },$$ for any $q\in S_P$ and $t\in\mathbb{R}_+$.\\
\end{proof}\ \\

\section{Strongly dynamics preserving property}\label{sdpp}
In this section, we define the limit maps of an Anosov representation over a closed subflow and investigate the properties of them.
In particular, we introduce the strongly dynamics preserving property and show that (3)$\Rightarrow$(4)$\Rightarrow$(2) in Theorem \ref{MT1}.\\

Let $\Gamma$ be a hyperbolic group and $S$ a fixed finite, symmetric generating set. Let $\rho$ be a representation of $\Gamma$ into $\mathrm{GL}(d,\mathbb{R})$.
Let $P$ be a closed, $\Gamma$-invariant subset of $\partial^{(2)}\Gamma$. Recall that we defined the flow spaces
$\widetilde{S_P}=P\times \mathbb{R}\subset \widetilde{U\Gamma} = \partial^{(2)}\Gamma\times \mathbb{R}$
and $S_P=\widetilde{S_P}/\Gamma$ in Section \ref{equivalence1}.
We let $\Pi_1(P) = \{x\in\partial \Gamma\ |\ \exists y\in \partial \Gamma\ \text{with} \ (x,y)\in P \}$ denote the projection of $P$ to the first coordinate, and similarly, let
$\Pi_2(P)=\{y\in\partial \Gamma\ |\ \exists x\in \partial \Gamma\ \text{with} \ (x,y)\in P \}$ denote the projection of $P$ to the second coordinate.\\

\subsection{Extending the sets $\Gamma_P^+$ and $\Gamma_P^-$} \ \\

The subsets $\Gamma_P^+$ and $\Gamma_P^-$ of $\Gamma$ are defined in Section \ref{kdom}. They consist of elements along geodesics in $\Cay(\Gamma,S)$ which have endpoints in $P$ and pass through $\id$. It will be convenient for us to discuss about the properties of limit maps if we extend the subsets $\Gamma_P^+$ and $\Gamma_P^-$ by elements along quasi-geodesics in $\Cay(\Gamma,S)$.\\

Firstly, notice that $\Gamma_P^+$ and $\Gamma_P^-$ have the following properties.

\begin{lemma}\label{inverseflowinverseelt}
    We denote $(\Gamma_P^+)^{-1}=\{\gamma^{-1} \ | \gamma\in\Gamma_P^+\  \}$. Then $(\Gamma_P^+)^{-1}= \Gamma_P^-$.
\end{lemma}
\begin{proof}
    For any $\gamma\in \Gamma_P^+$, there exists a geodesic $l:\mathbb{R}\rightarrow \Cay(\Gamma, S)$ with $l(0)=\id $, $(l(+\infty),l(-\infty))\in P$ and $l(n)=\gamma$ for a positive integer $n$. Then the geodesic $l'(\ \cdot \ ) = \gamma^{-1}l(\ \cdot\ + n)$ satisfies $l'(0)=\id $, $(l'(+\infty),l'(-\infty))\in P$ and $l'(-n)=\gamma^{-1}$. Hence $\gamma^{-1} \in\Gamma_P^-$. The other direction is similar.\\
\end{proof}

Recall that $\hat{P}=\{(y,x) \ |\ (x,y)\in P\}$.
Then $\Gamma_{\hat{P}}^+= \Gamma_{P}^-=(\Gamma_P^+)^{-1}$ by the lemma above.\\

Let $\lambda\geqslant 1 $ and $\epsilon,b\geqslant 0$ be three constants. We define $$Q_{P,(\lambda,\epsilon,b)}=\{ l:\mathbb{R}\rightarrow \Cay(\Gamma,S)\ |\ l\ \text{is a}\ (\lambda,\epsilon)\text{-quasi-geodesic with}\ $$ $$d(l(0),\id)\leqslant b\ \text{and}\ (l(+\infty),l(-\infty))\in P \}$$ and $$\Gamma^+_{P,(\lambda,\epsilon,b)}=\{ l(t)\ |\ l\in Q_{P,(\lambda,\epsilon,b)}, t\in \mathbb{R}_+\ \text{with}\ l(t)\in \Gamma \}.$$ Here $l(t)\in \Gamma$ means that $l(t)$ is a vertex in the Cayley graph. Then we have $\Gamma_P^+=\Gamma^+_{P,(1,0,0)}$, and $\Gamma^+_{P,(\lambda,\epsilon,b)} \subset \Gamma^+_{P,(\lambda',\epsilon',b')}$ if $\lambda\leqslant \lambda'$, $\epsilon\leqslant \epsilon'$ and $b\leqslant b'$. We also denote $Q^{+\infty}_{P,(\lambda,\epsilon,b)}= \{ l(+\infty)\in \partial\Gamma \ |\ l\in Q_{P,(\lambda,\epsilon,b)} \}.$\\

Follows from Proposition \ref{FTPs} and Observation \ref{SingluarvalueEstimate}, we have the following observation.

\begin{observation}\label{observationext}
    For any $\lambda\geqslant 1$ and $\epsilon, b\geqslant 0$, $\rho$ is $k$-dominated on $\Gamma_P^+$ if and only $\rho$ is $k$-dominated on $\Gamma^+_{P,(\lambda,\epsilon,b)}$.
\end{observation}

For a subset $T\subset \Gamma$, we denote the limit set of $T$ by $\Lambda(T)$, that is, 
$\Lambda(T)= \overline{T}\cap \partial\Gamma$, where $\overline{T}$ is the closure of $T$ in $\Gamma\cup \partial\Gamma$.

\begin{lemma}\label{preunionconverge} For any $\lambda\geqslant 1$ and $\epsilon,b\geqslant 0$,
    \begin{itemize}
        \item[(1).] $Q^{+\infty}_{P,(\lambda,\epsilon,b)}\subset\Lambda(\Gamma^+_{P,(\lambda,\epsilon,b)})\subset \Pi_1(P)$;
        \item[(2).] There exists $b'\geqslant 0$, such that $\Lambda(\Gamma^+_{P,(\lambda,\epsilon,b)})\subset Q^{+\infty}_{P,(1,0,b')}$;
        \item[(3).] $Q^{+\infty}_{P,(1,0,b)}=\Lambda(\Gamma^+_{P,(1,0,b)})$ is a closed set for any $b\geqslant 0$. 
    \end{itemize}
\end{lemma}

\begin{proof}
    The inclusion $Q^{+\infty}_{P,(\lambda,\epsilon,b)}\subset\Lambda(\Gamma^+_{P,(\lambda,\epsilon,b)})$ is obvious. Let $R_0\geqslant 0$ be the upper bound of the Hausdorff distance between two $(\lambda, \epsilon)$-quasi-geodesics given  and let $R\geqslant 0$ be the order preserving upper bound of the Hausdorff distance between two $(\lambda, \epsilon)$-quasi-geodesics by Proposition \ref{FTPs}. Let $\{\gamma_n\}$ be a sequence in $\Gamma^+_{P,(\lambda,\epsilon,b)}$ and $\gamma_n\rightarrow x\in \partial\Gamma$ as $n\rightarrow \infty$.  Then each $\gamma_n$ is on a $(\lambda,\epsilon)$-quasi-geodesic $l_n\in Q_{P,(\lambda,\epsilon,b)}$. Let $l'_n$ be a geodesic in  $\Cay(\Gamma,S)$ with $(l'_n(+\infty),l'_n(-\infty))=(l_n(+\infty),l_n(-\infty))$. Then the Hausdorff distance between $l_n$ and $l_n'$ are bounded by a constant $R_0$. So $l'_n$ intersects the ball of radius $b+R_0$ centered at $\id$. By reparametrizing $l'_n$ by a translation if necessary, we may assume $d(l_n'(0),\id)\leqslant b+R_0+1$.  Then there exists $s_n\geqslant 0$, such that $s_n\rightarrow +\infty$ as $n\rightarrow \infty$ and $d(\gamma_n,l_n'(s_n))\leqslant R+1$, so $l_n'(s_n)\rightarrow x$ as $n\rightarrow \infty$. Since $\{l_n'(0)\}$ is bounded, by taking a subsequence, we may assume $l_n'$ converges to a geodesic $l$, with $d(l(0),\id)\leqslant b+R_0+1$. Then $(l(+\infty),l(-\infty))\in P$ because $P$ is closed, and $l(+\infty)=x$. Hence $x\in Q^{+\infty}_{P,(1,0,b+R_0+1)}\subset \Pi_1(P)$, which proves (1) and (2).\\
    
    For (3), to check $Q^{+\infty}_{P,(1,0,b)}=\Lambda(\Gamma^+_{P,(1,0,b)})$ and closeness, we repeat the above argument by choosing $l_n=l_n'$.\\
\end{proof}

It is clear that $Q^{+\infty}_{P,(\lambda,\epsilon,b)}\subset Q^{+\infty}_{P,(\lambda',\epsilon',b')}$ and $\Lambda(\Gamma^+_{P,(\lambda,\epsilon,b)}) \subset \Lambda(\Gamma^+_{P,(\lambda',\epsilon',b')})$ when $\lambda\leqslant \lambda'$, 
$\epsilon\leqslant \epsilon'$ and $b\leqslant b'$ since $\Gamma^+_{P,(\lambda,\epsilon,b)} \subset \Gamma^+_{P,(\lambda',\epsilon',b')}$. Moreover, we have

\begin{lemma}\label{unionconverge}
    $$\bigcup_{b\geqslant 0} Q^{+\infty}_{P,(1,0,b)} = \bigcup_{\lambda\geqslant 1, \epsilon\geqslant 0, b\geqslant 0} Q^{+\infty}_{P,(\lambda,\epsilon,b)} = \bigcup_{\lambda\geqslant 1, \epsilon\geqslant 0, b\geqslant 0} \Lambda(\Gamma^+_{P,(\lambda,\epsilon,b)}) =\Pi_1(P).$$
\end{lemma}
\begin{proof}
It is clear that $$\bigcup_{b\geqslant 0} Q^{+\infty}_{P,(1,0,b)} \subset \bigcup_{\lambda\geqslant 1, \epsilon\geqslant 0, b\geqslant 0} Q^{+\infty}_{P,(\lambda,\epsilon,b)}\subset \bigcup_{\lambda\geqslant 1, \epsilon\geqslant 0, b\geqslant 0} \Lambda(\Gamma^+_{P,(\lambda,\epsilon,b)}) \subset \Pi_1(P).$$ Let $x\in \Pi_1(P)$ be a point on the boundary with $(x,y)\in P$. There exists a geodesic $l:\mathbb{R}\rightarrow \Cay(\Gamma,S)$ with $(l(+\infty),l(-\infty))=(x,y)$. Then $l\in Q^{+}_{P,(1,0,|l(0)|)}$, which implies $$\Pi_1(P) \subset \bigcup_{b\geqslant 0} Q^{+\infty}_{P,(1,0,b)}.$$
\end{proof}

Similar to the extension of $\Gamma^+_P$, we define $$\Gamma^-_{P,(\lambda,\epsilon,b)}=\{ l(-t)\ |\ l\in Q_{P,(\lambda,\epsilon,b)}, t\in \mathbb{R}_+\ \text{with}\ l(-t)\in \Gamma \}.$$ and $$Q^{-\infty}_{P,(\lambda,\epsilon,b)}= \{ l(-\infty)\in \partial\Gamma \ |\ l\in Q_{P,(\lambda,\epsilon,b)} \}.$$ By the same argument as in Lemma \ref{inverseflowinverseelt}, we have $\Gamma_{\hat{P},(\lambda,\epsilon,b)}^+= \Gamma_{P,(\lambda,\epsilon,b)}^-=(\Gamma_{P,(\lambda,\epsilon,b)}^+)^{-1}$ and $Q^{-\infty}_{P,(\lambda,\epsilon,b)}=Q^{+\infty}_{\hat{P},(\lambda,\epsilon,b)}$.\\
    
\subsection{The limit maps of $k$-Anosov representations over $S_P$}\ \\

We now assume $\rho: \Gamma \rightarrow \mathrm{GL}(d,\mathbb{R})$ is $k$-dominated on $\Gamma^+_{P,(\lambda,\epsilon,b)}$ for any $\lambda\geqslant 1$ and $\epsilon,b\geqslant 0$. We define the limit maps and discuss several properties of them, some of which will also be used in Section \ref{holdercts}. Then we show (3)$\Rightarrow$(4) in Theorem \ref{MT1}.\\

We firstly define the strongly dynamics preserving property for the limit maps. It was introduced by Canary--Zhang--Zimmer \cite{CZZ} in their study of Anosov representations of geometrically finite Fuchsian groups. We can generalize the concept to $k$-dominated representations on $\Gamma^+_{P,(\lambda,\epsilon,b)}$.\\

Let $\zeta^k:\Pi_1(P)\rightarrow \mathrm{Gr}_k(\mathbb{R}^d)$ and $\zeta_{d-k}:\Pi_2(P)\rightarrow \mathrm{Gr}_{d-k}(\mathbb{R}^d)$ be two $\rho$-equivariant maps (i.e., $\rho(\gamma)\zeta^k(x)=\zeta^k(\gamma x)$ and $\rho(\gamma)\zeta_{d-k}(y)=\zeta_{d-k}(\gamma y)$ for any $\gamma \in \Gamma$, $x\in\Pi_1(P)$ and $y\in \Pi_2(P)$), and the pair $(\zeta^k,\zeta_{d-k}):P\rightarrow \mathrm{Gr}_k(\mathbb{R}^d)\times \mathrm{Gr}_{d-k}(\mathbb{R}^d)$ is transverse (i.e., $\zeta^k(x)$ and $\zeta_{d-k}(y)$ are transverse for any $(x,y)\in P$) and continuous.
    
\begin{definition}\label{SDPdef}
    We say that a continuous, transverse pair $(\zeta^k,\zeta_{d-k})$ of two $\rho$-equivariant maps, is \emph{strongly dynamics preserving of index $k$ on $P$} if for any given constants $\lambda\geqslant 1$ and $\epsilon,b\geqslant 0$, points $x,y\in \partial\Gamma$, and a sequence $\{\gamma_n\}$ in $\Gamma^+_{P,(\lambda,\epsilon,b)}$, with $\gamma_n\rightarrow x$ and $\gamma_n^{-1}\rightarrow y$ as $n\rightarrow +\infty$, we have  $$\lim\limits_{n\rightarrow +\infty}\rho(\gamma_n)V=\zeta^k(x)$$ for any $V\in \mathrm{Gr}_k(\mathbb{R}^d)$ transverse to $\zeta_{d-k}(y)$.
\end{definition}

It is not hard to see that once there exists a pair $(\zeta^k,\zeta_{d-k})$ which satisfies the conditions in Definition \ref{SDPdef}, it is unique. We denote this unique pair by $(\xi^k,\xi_{d-k})$, called the pair of limit maps of $\rho$. Now we show the existence by construction.

\begin{proposition}\label{uniformconvergenceoflimitmaps}
    Let $\lambda\geqslant 1$ and $\epsilon,b\geqslant 0$ be constants and let $x$ be a point in $Q^{+\infty}_{P,(\lambda,\epsilon,b)}$. Let $l\in Q_{P,(\lambda,\epsilon,b)}$ with $l(+\infty)=x$. Let $\{\gamma_n\}$ denote all the elements in $\Gamma$ along $l([0,+\infty))$ in order. Then $\xi^k_{(\lambda,\epsilon,b)}(x)=\xi^k(x)=lim_{n\rightarrow +\infty} U_k(\rho(\gamma_n))$ is a well-defined map from $Q^{+\infty}_{P,(\lambda,\epsilon,b)}$ to $\mathrm{Gr}_k(\mathbb{R}^d)$. Moreover, the convergence of the above limit is uniform for any $l\in Q_{P,(\lambda,\epsilon,b)}$.
\end{proposition}

\begin{proof}
We check that $U_k(\rho(\gamma_n))$ is a Cauchy sequence. $$d(U_k(\rho(\gamma_n)),U_k(\rho(\gamma_{n+1}))) \leqslant \Vert \rho(\gamma_n)^{-1}\rho(\gamma_{n+1})\Vert \cdot\Vert \rho(\gamma_{n+1})^{-1}\rho(\gamma_{n})\Vert \cdot\dfrac{\sigma_{k+1}(\rho(\gamma_{n}))}{\sigma_k(\rho(\gamma_{n}))}$$ by Lemma \ref{estsingularvalue} (3). The first term $\Vert \rho(\gamma_n)^{-1}\rho(\gamma_{n+1})\Vert \cdot\Vert \rho(\gamma_{n+1})^{-1}\rho(\gamma_{n})\Vert $ is uniformly bounded since for all $n$, $\gamma_{n+1}^{-1}\gamma_n$ lies in the finite set $\{\gamma\in\Gamma \ |\ |\gamma|\leqslant \lambda+\epsilon\}$. The second term $\frac{\sigma_{k+1}(\rho(\gamma_{n}))}{\sigma_k(\rho(\gamma_{n}))}$ converges to $0$ since $\gamma_n\in\Gamma^+_{P,(\lambda,\epsilon,b)}$ and $\rho$ is $k$-dominated on $\Gamma^+_{P,(\lambda,\epsilon,b)}$. This shows uniform convergence.\\

To show $\xi^k$ is well-defined, let $l'\in Q_{P,(\lambda,\epsilon,b)}$ be another $(\lambda,\epsilon)$-quasi-geodesic with $l'(+\infty)=x$. By Proposition \ref{FTP}, let $R_0\geqslant 0$ be the upper bound of the Hausdorff distance between two $(\lambda, \epsilon+2b)$-quasi-geodesics. Then the Hausdorff distance between $l([0,+\infty))$ and $l'([0,+\infty))$ is bounded by a constant $R_0$. Let $\gamma'_n$ be an element on $l'([0,+\infty))$ such that $d(\gamma_n,\gamma'_n)\leqslant R_0$ for each $n$. Then apply Lemma \ref{estsingularvalue} (3) for $\gamma_n$ and $\gamma'_n$, we have $\xi^k(x)$ is well-defined.\\
\end{proof}

\begin{remark}
    If we have constants $\lambda\leqslant \lambda'$, $\epsilon\leqslant \epsilon'$ and $b\leqslant b'$, $\Gamma^+_{P,(\lambda,\epsilon,b)} \subset \Gamma^+_{P,(\lambda',\epsilon',b')}$, $Q^{+\infty}_{P,(\lambda,\epsilon,b)}\subset Q^{+\infty}_{P,(\lambda',\epsilon',b')}$, then we have $\xi^k_{(\lambda',\epsilon',b')}$ and $\xi^k_{(\lambda,\epsilon,b)}$ agree on $Q^{+\infty}_{P,(\lambda,\epsilon,b)}$. By Lemma \ref{unionconverge}, $\xi^k$ is well-defined on $\Pi_1(P)$.
\end{remark}

We may define the limit map $\xi_{d-k}:\Pi_2(P)\rightarrow \mathrm{Gr}_{d-k}(\mathbb{R}^d)$ in the following way. We write $\xi^k=\xi^k_P$ to emphasize the limit map is defined for $\rho$ which is $k$-dominated on $\Gamma^+_P$. Since $\rho$ is also $(d-k)$-dominated on $\Gamma^+_{\hat{P}}$, by the above discussion, we defined the limit map $\xi_{d-k}=\xi_{d-k,P}=\xi^{d-k}_{\hat{P}}$.\\

Then our goal is to check the limit maps $\xi^k$ and $\xi_{d-k}$ satisfy the conditions in Definition \ref{SDPdef}, i.e., the following theorem.

\begin{theorem}\label{1to4mainthm1}
    If $\rho$ is $k$-dominated on $\Gamma^+_P$, then the limit maps $\xi^k:\Pi_1(P)\rightarrow \mathrm{Gr}_k(\mathbb{R}^d)$ and $\xi_{d-k}:\Pi_2(P)\rightarrow \mathrm{Gr}_{d-k}(\mathbb{R}^d)$ are $\rho$-equivariant, and $(\xi^k,\xi_{d-k}):P\rightarrow \mathrm{Gr}_k(\mathbb{R}^d)\times \mathrm{Gr}_{d-k}(\mathbb{R}^d)$ is continuous, transverse and strongly dynamics preserving of index $k$ on $P$.
\end{theorem}

To prove the theorem, we need the following propositions.

\begin{proposition}\label{ctspiP}
    $\xi^k$ is continuous on $Q^{+\infty}_{P,(\lambda,\epsilon,b)}$ for any $\lambda\geqslant 1 $ and $\epsilon,b\geqslant 0$.
\end{proposition}

\begin{proof}
    By Lemma \ref{preunionconverge} (1) and (2), it is sufficient to show  $\xi^k$ is continuous on $Q^{+\infty}_{P,(1,0,b)}$ for any $b\geqslant 0$. Let $x_n\in Q^{+\infty}_{P,(1,0,b)}$ that converges to $x\in Q^{+\infty}_{P,(1,0,b)}$ as $n\rightarrow +\infty$. We claim that any subsequence of $x_n$ has a subsequence such that $\xi^k(x_n)\rightarrow \xi^k(x)$.\\
    
    Let $l_n\in Q_{P,(1,0,b)}$ be a geodesic with $l_n(+\infty)=x_n$. Up to a bounded reparametrization by translation for each $l_n$, we may assume $l_n(m)\in\Gamma$ for any $m\in\mathbb{Z}$. We also assume that up to a subsequence, $l_n$ converge to a geodesic $l$ with $d(l(0),\id)\leqslant b$ and $l(+\infty)=x$. Therefore, up to a subsequence again, we may moreover assume that up to a subsequence, $l_n(0)=l(0)$ for all $n$. Then $$\xi^k(x_n)= \lim\limits_{m\rightarrow +\infty} U_k(\rho(l_n(m))),$$ and $$\xi^k(x)= \lim\limits_{m\rightarrow +\infty} U_k(\rho(l(m))).$$ Since the convergence is uniform for all $l_n$ and $l$, it is sufficient to find a sequence of positive integers $m_n$, such that $m_n\rightarrow +\infty$ and $$\lim\limits_{n\rightarrow +\infty} d(U_k(\rho(l_n(m_n))),U_k(\rho(l(m_n))))=0.$$
    
    Apply Lemma \ref{estsingularvalue} (3), we have $$d(U_k(\rho(l_n(m_n))),U_k(\rho(l(m_n)))) $$ $$\leqslant \Vert \rho(l_n(m_n)^{-1}l(m_n))\Vert \cdot\Vert \rho(l(m_n)^{-1}l_n(m_n))\Vert \cdot\dfrac{\sigma_{k+1}(\rho(l(m_n)))}{\sigma_{k}(\rho(l(m_n)))}$$ Since $l_n\rightarrow l$ as $n\rightarrow +\infty$ uniformly on any compact set, then for any $r\in \mathbb{Z}_+$, there exists $N_p>0$, such that for any $n\geqslant N_r$, $$ \max\{d(l_n(t),l(t)) \ | \  t\in [0,r]\} \leqslant \dfrac{1}{r}.$$ For any $r>0$, we may choose $N_r$, such that $\{N_r\}$ is an increasing sequence with $N_r\rightarrow +\infty$ as $r\rightarrow +\infty$. Let $m_n=r$ if $n\in [N_r,N_{r+1})$. Then $m_n\rightarrow +\infty$ as $n\rightarrow +\infty$ and $$d(l_n(m_n),l(m_n))\leqslant \dfrac{1}{m_n}.$$ Hence we have that $|l_n(m_n)^{-1}l(m_n)|=d(l_n(m_n),l(m_n))\rightarrow 0$ as $n\rightarrow +\infty$. Therefore $\Vert \rho(l_n(m_n)^{-1}l(m_n))\Vert \cdot\Vert \rho(l(m_n)^{-1}l_n'(m_n))\Vert $ is bounded, and $\dfrac{\sigma_{k+1}(\rho(l(m_n)))}{\sigma_{k}(\rho(l(m_n)))}$ converges to $0$ by $\rho$ is $k$-dominated on $\Gamma^+_{P,(\lambda,\epsilon,b+C)}$.\\
\end{proof}

\begin{remark}
    We will show that $\xi^k$ is, moreover, Hölder continuous on $Q^{+\infty}_{P,(\lambda,\epsilon,b)}$ in section \ref{holdercts} by a similar, but more precise computation.
\end{remark}

However, $\xi^k$ may not be continuous on $\Pi_1(P)$. Here is a counter-example.

\begin{example}\label{exampledisctslimitmap}
    Let $\rho$ be a representation of the free group $F_2=F(a,b)$ generated by $\{a,b\}$ into $\mathrm{GL}(2,\mathbb{R})$ with 
    $\rho(a)=\left(\begin{matrix}
        2 & 0\\
        0 & \frac{1}{2}
    \end{matrix} \right)$
    and
    $\rho(b)=\left(\begin{matrix}
        0 & 1\\
        -1 & 0
    \end{matrix}\right)$.
Then we have $\rho$ is $1$-Anosov over $S_P$ with $P=F_2\cdot\{(a^+,a^-)\}$. By computation, we have $$\xi^k(a^+)= \lim\limits_{n\rightarrow +\infty} U_1(\rho(a^n))= \mathrm{span}\{(1,0)\}$$ while $\lim\limits_{m\rightarrow +\infty}a^mb\cdot a^+= a^+$ and $$\xi^k(a^mb\cdot a^+)= \lim\limits_{n\rightarrow +\infty} U_1(\rho(a^mba^n))= \mathrm{span}\{(0,1)\}.$$ Hence $\xi^k$ is not continuous on $\Pi_1(P)$.
\end{example}

\begin{proposition}[a weaker $P_k$-Cartan Property]\label{PkCartan}
    Let $\gamma_n$ be a sequence in $\Gamma$ and $\gamma_n\rightarrow x\in \Pi_1(P)$. If we can find constants $\lambda\geqslant 1, \epsilon, b\geqslant 0$ such that  $\gamma_n\in \Gamma^+_{P,(\lambda,\epsilon,b)}$ for all $n\in \mathbb{N}$, then $$\xi^k(x)=\lim\limits_{n\rightarrow +\infty} U_k(\rho(\gamma_n)).$$
\end{proposition}
\begin{proof}
   Each $\gamma_n$ lies on a $(\lambda,\epsilon)$-quasi-geodesic $l_n\in Q_{P,(\lambda,\epsilon,b)}$ with $\gamma_n=l_n(t_n)$, where $t_n>0$ and $t_n\rightarrow +\infty$ as $n\rightarrow +\infty$, so $l_n(+\infty)\rightarrow x$ as $n\rightarrow +\infty$. We may also assume $x\in Q^{+\infty}_{P,(\lambda,\epsilon,b)}$ by enlarging the constant $b$. Since $$d(\xi^k(x), U_k(\rho(\gamma_n)))\leqslant d(\xi^k(x),\xi^k(l_n(+\infty)))+d(\xi^k(l_n(+\infty)), U_k(\rho(\gamma_n))),$$ $d(\xi^k(x),\xi^k(l_n(+\infty)))$ converges to $0$ as $n\rightarrow +\infty$ by the continuity from Proposition \ref{ctspiP} of $\xi^k$, and $d(\xi^k(l_n(+\infty)), U_k(\rho(\gamma_n)))=d(\xi^k(l_n(+\infty)), U_k(\rho(l_n(t_n))))$ converges to $0$ as $n\rightarrow +\infty$ by the uniform property in Proposition \ref{uniformconvergenceoflimitmaps}, then $$\xi^k(x)=\lim\limits_{n\rightarrow +\infty} U_k(\rho(\gamma_n)).$$
\end{proof}

\begin{remark}
    The condition $\gamma_n\in \Gamma^+_{P,(\lambda,\epsilon,b)}$ for all $n\in \mathbb{N}$ allows us to use the uniform convergence property indicated in Proposition \ref{uniformconvergenceoflimitmaps}. In Example \ref{exampledisctslimitmap}, although we have $a^mb\cdot a^+\rightarrow a^+$ as $m\rightarrow \infty$, we cannot find a uniform $Q^{+\infty}_{P,(\lambda,\epsilon,b)}$ that contains $a^mb\cdot a^+$ for all $m\in \mathbb{N}$.
\end{remark}

Together with the following lemma, we will be ready to give the proof of Theorem \ref{1to4mainthm1}.

\begin{lemma}[\cite{CZZ} Lemma 2.2]\label{lemmaczz}
    Let $V_0\in \mathrm{Gr}_k(\mathbb{R}^d)$, $W_0\in \mathrm{Gr}_{d-k}(\mathbb{R}^d)$ and $\{A_n\}$ a sequence in $\mathrm{GL}(d,\mathbb{R})$. Then the following are equivalent.\\
    (1). $A_n\cdot V \rightarrow V_0$ as $n\rightarrow +\infty$ for any $V\in \mathrm{Gr}_k(\mathbb{R}^d)$ that transverse to $W_0$ with the convergence uniform on any compact subset.\\
    (2). $\lim\limits_{n\rightarrow +\infty} \dfrac{\sigma_{k+1}(A_n)}{\sigma_k(A_n)}=0$, $U_k(A_n)\rightarrow V_0$ and $U_{d-k}(A_n^{-1})\rightarrow W_0$ as $n\rightarrow +\infty$.
\end{lemma}

\begin{proof}[Proof of Theorem \ref{1to4mainthm1}]
    $\xi^k$ is $\rho$-equivariant on $\Pi_1(P)$ and $\xi_{d-k}$ is $\rho$-equivariant on $\Pi_2(P)$ by the definition of limit maps and Lemma \ref{estsingularvalue} (4).\\

    We show that $(\xi^k,\xi_{d-k}):P\rightarrow \mathrm{Gr}_k(\mathbb{R}^d)\times \mathrm{Gr}_{d-k}(\mathbb{R}^d)$ is continuous on $P$. Let $(x_i, y_i)$ be a sequence in $P$ and $(x,y)\in P$ such that $(x_i,y_i)\rightarrow (x,y)$ as $i\rightarrow +\infty$. Let $l_i$ be a geodesic in $\Cay(\Gamma, S)$ with $(l_i(+\infty),l_i(-\infty))=(x_i,y_i)$. Up to a subsequence and reparametrization, $l_i$ converges to $l$, a geodesic with $(l(+\infty), l(-\infty))=(x,y)$. Then for $i$ large enough, $l_i\in Q_{P,(1,0,|l(0)|+1)}$. Then $(\xi^k(x_i),\xi_{d-k}(y_i))\rightarrow (\xi^k(x),\xi_{d-k}(y))$ as $i\rightarrow +\infty$  by the continuity of $\xi^k$ on $Q^{+\infty}_{P,(1,0,|l(0)|+1)}$ and the continuity of $\xi_{d-k}$ on $Q^{-\infty}_{P,(1,0,|l(0)|+1)}$.\\

    Then we show that $(\xi^k,\xi_{d-k})$ is transverse on $P$. Let $l$ be a geodesic in $\Cay(\Gamma,S)$ with $(l(+\infty),l(-\infty))=(x,y)$. Since $P$ is $\Gamma$-invariant and $\xi^k$, $\xi_{d-k}$ are $\rho$-equivariant, we may assume $l(0)=\id$ by a bounded reparametrization by a translation and replacing $l$ by $l(0)^{-1}\cdot l$. Then $\xi^k(x)$ and $\xi_{d-k}(y)$ are transverse following from Lemma 2.5 in \cite{BPS}.\\

    By Proposition \ref{PkCartan} and Lemma \ref{lemmaczz}, $(\xi^k,\xi_{d-k})$ is strongly dynamics preserving of index $k$ on $P$.\\
\end{proof}

\subsection{Strongly dynamics preserving limit maps induce dominated splitting}\ \\

We now prove (4)$\Rightarrow$(2) in Theorem \ref{MT1}.

\begin{theorem}\label{sdpimpliesds}
    Let $\rho: \Gamma \rightarrow \mathrm{GL}(d,\mathbb{R})$ be a representation.
    If there exist two $\rho$-equivariant maps $\zeta^k:\Pi_1(P)\rightarrow \mathrm{Gr}_k(\mathbb{R}^d)$ and $\zeta_{d-k}:\Pi_2(P)\rightarrow \mathrm{Gr}_{d-k}(\mathbb{R}^d)$, such that
    $(\zeta^k,\zeta_{d-k}):P\rightarrow \mathrm{Gr}_k(\mathbb{R}^d)\times \mathrm{Gr}_{d-k}(\mathbb{R}^d)$ is continuous, transverse and
    strongly dynamics preserving of index $k$ on $P$,
    then $E_\rho$ admits a $k$-dominated splitting over $S_P$.
\end{theorem}
\begin{remark}
    In the theorem,
    since $E_\rho$ admitting a $k$-dominated splitting over $S_P$ implies $\rho$ is $\rho$ is $k$-dominated on $\Gamma^+_P$,
    the pair of limit maps $(\xi^k,\xi_{d-k})$ is defined for $\rho$. By the uniqueness,
    $(\zeta^k,\zeta_{d-k})=(\xi^k,\xi_{d-k})$.
\end{remark}
\begin{proof}[Proof of Theorem\ref{sdpimpliesds}]
    Let $K_0\subset \widetilde{S_P}$ be a compact subset and let $p\in K_0$ be a fixed base point. Let $\tau_p:\Cay(\Gamma, S)\rightarrow \widetilde{U\Gamma}$ be the orbit map, which is a quasi-isometry. Let $K\supset K_0$ be the compact subset given by Lemma \ref{equiesti}, let $\{q_n\}\subset S_P$ be an arbitrary sequence and let $z_n=(x_n,y_n,s_n)\in K$ be a lift of $q_n$ for each $n$. Then $(x_n,y_n)\in P$. Let $v_n\in \zeta^k(x_n)$ and $w_n\in \zeta_{d-k}(y_n)$ be unit vectors for each $n$. Let $t_n\in\mathbb{R}_+$ such that $t_n\rightarrow +\infty$ as $n\rightarrow +\infty$. Recall the two Riemannian metrics on $\widetilde{E}$ given by Notation \ref{twonorm}. Let $C>0$ be a constant such that, $C^{-1}\Vert \cdot\Vert \leqslant \Vert \cdot\Vert_0\leqslant C\Vert \cdot\Vert ,$ on $K$. We claim that $$\lim\limits_{n\rightarrow +\infty} \dfrac{ \Vert \psi_{q_n}^{t_n}(v_n)\Vert_{\phi^{t_n}(q_n)} }{ \Vert \psi_{q_n}^{t_n}(w_n)\Vert_{\phi^{t_n}(q_n)}} = 0.$$

    By Lemma \ref{equiesti}, there exists a finite set $B\subset \Gamma$, such that for each pair $z_n$ and $t_n$, there exists $\gamma_n\in \{\beta\omega\beta^{-1}\ |\ \beta\in B, \omega\in\Gamma_P^+\}$, $\gamma_n^{-1}\widetilde{\phi}^{t_n}(z_n)\in K$. Let $L_n$ be a geodesic in $\Cay(\Gamma,S)$ with endpoints $(\gamma_n^{-1}x_n, \gamma_n^{-1}y_n)$. Then the Hausdorff distance between $\tau_p(L_n)$ and $\{ (\gamma_n^{-1}x_n,\gamma_n^{-1}y_n,t) \ |\ t\in\mathbb{R}\}$ is uniformly bounded for all $n$ by Proposition \ref{FTP}. Since $\tau_p$ is a quasi-isometry and $\gamma_n^{-1}\widetilde{\phi}^{t_n}(z_n)\in K$, the distance between $L_n$ and $\id$ in the Cayley graph is uniformly bounded for all $n$. Hence by enlarging $\lambda,\epsilon$ and $b$, we have $L_n\in  Q_{P,(\lambda,\epsilon,b)}$, then $\gamma_n^{-1}x_n\in Q^{+\infty}_{P,(\lambda,\epsilon,b)}$ and $\gamma_n^{-1}y_n\in Q^{-\infty}_{P,(\lambda,\epsilon,b)}$. Then we have
    \begin{align*}
    \dfrac{\Vert \psi_{q_n}^{t_n}(v_n)\Vert_{\phi^{t_n}(q_n)}}{\Vert \psi_{q_n}^{t_n}(w_n)\Vert_{\phi^{t_n}(q_n)}}
    & = \dfrac{\Vert \widetilde{\psi}_{z_n}^{t_n}(v_n)\Vert_{\widetilde{\phi}^{t_n}(z_n)}}{\Vert \widetilde{\psi}_{z_n}^{t_n}(w_n)\Vert_{\widetilde{\phi}^{t_n}(z_n)}}
     = \dfrac{\Vert \rho(\gamma_n^{-1}) \widetilde{\psi}_{z_n}^{t_n}(v_n)\Vert_{\gamma_n^{-1}\widetilde{\phi}^{t_n}(z_n)}}{\Vert \rho(\gamma_n^{-1})  \widetilde{\psi}_{z_n}^{t_n}(w_n)\Vert_{\gamma_n^{-1}\widetilde{\phi}^{t_n}(z_n)}}\\
    &
    \leqslant C^2\cdot\dfrac{\Vert \rho(\gamma_n^{-1}) \widetilde{\psi}_{z_n}^{t_n}(v_n)\Vert_0}{\Vert \rho(\gamma_n^{-1})  \widetilde{\psi}_{z_n}^{t_n}(w_n)\Vert_0} = C^2\cdot\dfrac{\Vert \rho(\gamma_n^{-1})  v_n\Vert_0}{\Vert \rho(\gamma_n^{-1})  w_n\Vert_0}.
    \end{align*}
    
    To prove the claim, we now need only to show $$\lim\limits_{n\rightarrow +\infty} \dfrac{\Vert \rho(\gamma_n^{-1})  v_n\Vert_0}{\Vert \rho(\gamma_n^{-1})  w_n\Vert_0}= 0.$$

    Up to a subsequence, we assume $z_n\rightarrow z=(x,y,s)$ and $\gamma_n^{-1} \widetilde{\phi}^{t_n}(z_n)\rightarrow z'=(x',y',s')$ as $n\rightarrow +\infty$. Notice that $d(\widetilde{\phi}^{t_n}(z_n),\gamma_n z_n) = d(\gamma_n^{-1}\widetilde{\phi}^{t_n}(z_n),z_n)$ is bounded and $t_n\rightarrow +\infty$ as $n\rightarrow +\infty$. We have $\widetilde{\phi}^{t_n}(z_n) \rightarrow x$ as $n\rightarrow +\infty$ and hence $\gamma_n z_n \rightarrow x$ as $n\rightarrow +\infty$. On the other hand, $\widetilde{\phi}^{-t_n}(\gamma_n^{-1}\widetilde{\phi}^{t_n}(z_n)) = \gamma_n^{-1}z_n\rightarrow y'$ as $n\rightarrow +\infty$ since the flow and $\Gamma$-action commute. Therefore, we have $\gamma_n\rightarrow x$ and $\gamma_n^{-1}\rightarrow y'$ as $n\rightarrow +\infty$.\\

    For all $n$, we decompose $v_n=v_n^1+v_n^2$ with $v_n^1\in U_k(\rho(\gamma_n))$ and $v_n^2\in U_k(\rho(\gamma_n))^\perp$. Then $$\Vert \rho(\gamma_n^{-1})  v_n\Vert^2_0 =\Vert \rho(\gamma_n^{-1})  v_n^1\Vert^2_0 +\Vert \rho(\gamma_n^{-1})  v_n^2\Vert^2_0 ,$$ so it is sufficient to prove  $$\limsup\limits_{n\rightarrow +\infty} \dfrac{\Vert \rho(\gamma_n^{-1})  v_n^1\Vert_0}{\Vert \rho(\gamma_n^{-1})  w_n\Vert_0}=0 \ \text{and} \ \limsup\limits_{n\rightarrow +\infty} \dfrac{\Vert \rho(\gamma_n^{-1})  v_n^2\Vert_0}{\Vert \rho(\gamma_n^{-1})  w_n\Vert_0}=0.$$ Since $v_n^1\in U_k(\rho(\gamma_n))$, we have $$\limsup\limits_{n\rightarrow +\infty} \dfrac{\Vert \rho(\gamma_n^{-1})v_n^1\Vert_0}{\sigma_{d-k+1}(\rho(\gamma_n))^{-1}} \leqslant \Vert v_n^1\Vert_0 \leqslant \Vert v_n \Vert_0 \leqslant C.$$ Since $\gamma_n\in \Gamma^+_{P,(\lambda,\epsilon,b)}$, by the strongly dynamics preserving property and Lemma \ref{lemmaczz}, we have $$\lim\limits_{n\rightarrow +\infty} U_k(\rho(\gamma_n))=\zeta^k(x)\ \text{and}\ \lim\limits_{n\rightarrow +\infty} U_{d-k}(\rho(\gamma_n^{-1}))=\zeta_{d-k}(y').$$ We then have $$\lim\limits_{n\rightarrow +\infty} \angle( U_{d-k}(\rho(\gamma_n^{-1})),\rho(\gamma_n^{-1})w_n)=0$$ because $\rho(\gamma_n)^{-1}w_n\in \zeta_{d-k}(\gamma_n^{-1}y_n)$ and $\zeta_{d-k}$ is continuous on $Q^{-\infty}_{P,(\lambda,\epsilon,b)}$. Then we have $$\liminf\limits_{n\rightarrow +\infty} \dfrac{\Vert \rho(\gamma_n^{-1})w_n\Vert_0}{\sigma_{d-k}(\rho(\gamma_n)^{-1})} \geqslant \liminf\limits_{n\rightarrow +\infty}\Vert w_n\Vert_0 \geqslant \dfrac{1}{C}.$$ Hence we get $$\limsup\limits_{n\rightarrow +\infty} \dfrac{\Vert \rho(\gamma_n^{-1})  v_n^1\Vert_0}{\Vert \rho(\gamma_n^{-1})  w_n\Vert_0} \cdot \dfrac{\sigma_{d-k}(\rho(\gamma_n^{-1}))}{\sigma_{d-k+1}(\rho(\gamma_n^{-1}))}\leqslant C^2.$$\\

    Since $(\zeta^k,\zeta_{d-k})$ is strongly dynamics preserving of index $k$ on $P$ and by Lemma \ref{lemmaczz}, we have $$\lim\limits_{n\rightarrow +\infty} \dfrac{\sigma_{k+1}(\rho(\gamma_n))}{\sigma_{k}(\rho(\gamma_n))} = 0$$ and hence $$\lim\limits_{n\rightarrow +\infty} \dfrac{\sigma_{d-k+1}(\rho(\gamma_n^{-1}))}{\sigma_{d-k}(\rho(\gamma_n^{-1}))} = 0.$$ Then $$\limsup\limits_{n\rightarrow +\infty} \dfrac{\Vert \rho(\gamma_n^{-1})  v_n^1\Vert_0}{\Vert \rho(\gamma_n^{-1})  w_n\Vert_0}=0.$$

    Next, we show that $$\limsup\limits_{n\rightarrow +\infty} \dfrac{\Vert \rho(\gamma_n^{-1})  v_n^2\Vert_0}{\Vert \rho(\gamma_n^{-1})  w_n\Vert_0}=0.$$ Suppose this is not the case, then up to a subsequence, $$\limsup\limits_{n\rightarrow +\infty} \dfrac{\Vert \rho(\gamma_n^{-1})  v_n^2\Vert_0}{\Vert \rho(\gamma_n^{-1})  v_n^1\Vert_0}=+\infty.$$ Then the two sequences of unit vectors $\rho(\gamma_n^{-1})  v_n^2 / \Vert \rho(\gamma_n^{-1})  v_n^2\Vert_0$ and $\rho(\gamma_n^{-1})  v_n / \Vert \rho(\gamma_n^{-1})  v_n\Vert_0$ will converge to the same unit vector $u$.\\

    For the first sequence, since $\rho(\gamma_n^{-1})  v_n^2\in \rho(\gamma_n^{-1})  U_k(\rho(\gamma_n))^\perp =U_{d-k}(\rho(\gamma_n^{-1}))$ and $\lim\limits_{n\rightarrow +\infty} U_{d-k}(\rho(\gamma_n^{-1})) =\zeta_{d-k}(y')$. Hence $u\in \zeta_{d-k}(y')$.\\

    For the second sequence, since $\rho(\gamma_n^{-1})  v_n\in \zeta^k(\gamma_n^{-1}x_n)$ as $\zeta^k$ is $\rho$-equivariant, $u\in \lim\limits_{n\rightarrow +\infty} \zeta^k(\gamma_n^{-1}x_n)
    =\zeta^k(x')$ because $(\zeta^k,\zeta_{d-k})$ is continuous on $P$ and $(\gamma^{-1}x_n,\gamma^{-1}y_n)\rightarrow (x',y')$ as $n\rightarrow +\infty$. However, we have $\zeta^k(x')$ and $\zeta_{d-k}(y')$ are transverse since $(x',y')\in P$, which gives a contradiction. Therefore, we have $$\lim\limits_{n\rightarrow +\infty}\dfrac{\Vert \psi_{z_n}^{t_n}(v_n)\Vert_{\phi^{t_n}(z_n)}}{\Vert \psi_{z_n}^{t_n}(w_n)\Vert_{\phi^{t_n}(z_n)}} \leqslant\lim\limits_{n\rightarrow +\infty} \dfrac{\Vert \rho(\gamma_n^{-1})  v_n\Vert_0}{\Vert \rho(\gamma_n^{-1})  w_n\Vert_0}=0.$$ The claim is proved.\\

    Since $(\zeta^k,\zeta_{d-k})$ is continuous, $\rho$-equivariant and transverse on $P$, it defines a splitting of $E|_{S_P}$. Denote it by
    $$E|_{S_P}= E^k\oplus E_{d-k}.$$
    By the compactness of $K$ and the claim, there exists a constant $T>0$, such that $$\dfrac{\Vert \psi_{z}^{t}(v)\Vert_{\phi^{t}(q)}}{\Vert \psi_{z}^{t}(w)\Vert_{\phi^{t}(q)}}= \dfrac{\Vert \widetilde{\psi}_z^t(v)\Vert_{\widetilde{\phi}^{t}(z)}}{\Vert \widetilde{\psi}_{z}^{t}(w)\Vert_{\widetilde{\phi}^{t}(z)}} \leqslant \dfrac{1}{2}\cdot \dfrac{\Vert v\Vert_{\widetilde{\phi}^{t}(z)}}{\Vert w\Vert_{\widetilde{\phi}^{t}(z)}} = \dfrac{1}{2}\cdot \dfrac{\Vert v\Vert_{q}}{\Vert w\Vert_{q}}$$ for any $t>T$, $q\in S_P$ with a lift $z=(x,y,s)\in K$, any $v\in E^k_{z}$ with a lift $v\in \zeta^k(x)$ and any $w\in E_{d-k,z}$ with a lift $w\in \zeta_{d-k}(y)$.\\

    Thus, $$\dfrac{\Vert \psi_{z}^{t}(v)\Vert_{\phi^{t}(z)}}{\Vert \psi_{z}^{t}(w)\Vert_{\phi^{t}(z)}}\leqslant \left(\dfrac{1}{2}\right)^{\frac{t}{T}-1}\cdot \dfrac{\Vert v\Vert_{z}}{\Vert w\Vert_{z}}$$ for any $t\in \mathbb{R}_+$, $v\in E^k_{z}$ and $w\in E_{d-k,z}$. Then $E|_{S_P}= E^k\oplus E_{d-k}$ is a $k$-dominated splitting.\\
\end{proof}\ \\

\section{Hölder continuity of the limit maps}\label{holdercts}
In this section, we show that the limit maps defined by Proposition \ref{uniformconvergenceoflimitmaps} are Hölder continuous in a weaker sense through a purely algebraic way, i.e., Theorem \ref{MT2} (2). To do this, we need to introduce a coarse projection in $\delta$-hyperbolic spaces and several operations on quasi-geodesics.\\

\subsection{More about $\delta$-hyperbolic spaces}\ \\

Let $(X,d)$ be a $\delta$-hyperbolic metric space with $\delta\geqslant 0$ a constant. We recall the definition and properties of the (visual) metric on the Gromov boundary. Then we introduce a way to define a coarse projection of a point $x\in\partial X$ to a geodesic $l:\mathbb{R}\rightarrow X$ with $x\ne l(+\infty),l(-\infty)$.\\

Let $p\in X$ be a fixed base point. For any two points $x,y\in X$, the Gromov product of $x$ and $y$ with base point $p$ is defined by $$(x\cdot y)_p=\dfrac{1}{2}(d(x,p)+d(y,p)-d(x,y)).$$ The Gromov product of two points $x,y$ in $X\cup \partial X$ with base point $p$ can be defined by limit process. That is, $$(x\cdot y)_p = sup \liminf\limits_{m,n\rightarrow \infty} (x_m,y_n)_p$$ where the supremum is taken from all sequences $\{x_n\}$ and $\{y_n\}$ in $X$ with $x_n\rightarrow x$ and $y_n\rightarrow y$ as $n\rightarrow \infty$. It is not hard to see that $(x\cdot y)_p\geqslant 0$ for any $x,y\in X\cup \partial X$ and $(x\cdot y)_p=\infty$ if and only if $x=y\in\partial X$.

\begin{lemma}[\cite{BH} Remarks 3.17 (5)]\label{coarselyproj}
    Let $x,y\in\partial X$ and $\{x_n\}$ and $\{y_n\}$ are two sequences in $X$ with $x_n\rightarrow x$ and $y_n\rightarrow y$ as $n\rightarrow \infty$. Then $$ (x\cdot y)_p-2\delta\leqslant \liminf\limits_{m,n\rightarrow \infty} (x_m,y_n)_p \leqslant (x\cdot y)_p.$$
\end{lemma}

Following \cite{BH}, we define a pseudo-metric on $\partial X$ with a constant $\kappa>0$ by $$d_\kappa (x,y)=inf \{\ \sum_{i=1}^n e^{-\kappa (x_{i-1}\cdot x_i)_p} \ |\ n\in\mathbb{N}, \{x_i\}_i\subset \partial X\  \text{with}\ x=x_0, y=x_n \ \}$$ for any $x,y\in\partial X$.

\begin{proposition}[\cite{BH} Part III, Proposition 3.21]\label{visualmetric}
    If $\kappa \leqslant \dfrac{1}{4\delta}log2$, then $d_\kappa$ is a metric on $\partial X$. Moreover, $$(3-2e^{2\delta\kappa}) e^{-\kappa(x\cdot y)_p} \leqslant d_\kappa (x,y)\leqslant e^{-\kappa (x\cdot y)_p}.$$    
\end{proposition}\ \\
\begin{definition}
    Let $l:\mathbb{R}\rightarrow X$ be a geodesic and $x\in\partial X - \{ l(+\infty),l(-\infty) \}$. We denote the real number $t_l(x)=(l(+\infty)\cdot x)_{l(0)}-(l(-\infty)\cdot x)_{l(0)}$. We define the (coarse) projection to $l$ to be the map 
    \begin{align*}
        p_l: \partial X - \{l(+\infty),l(-\infty)\} & \longrightarrow l(\mathbb{R}) \\ x & \longmapsto p_l(x)=l(t_l(x)).
    \end{align*}
\end{definition}

It is not hard to see that $p_l$ is $\mathrm{Isom}(X)$-equivariant, that is $\gamma\cdot p_l(x)= p_{\gamma\cdot l} (\gamma\cdot x)$.\\

As defined, the projection depends not only on the image of $l$, but also the parametrization of $l$. If we consider the reparametrized geodesic $l'(t)=l(t+k)$ for some $k\in\mathbb{R}$, we can check that $d(p_l(x),p_{l'}(x))\leqslant 4\delta$ by the following proposition. Therefore, $p_l$ is coarsely well-defined when we regard $l$ as only the image of a geodesic.

\begin{proposition}\label{coarselyprojrmk}
    Let $l:\mathbb{R}\rightarrow X$ be a geodesic and $l'(t)=l(t+k)$ for a constant $k\in\mathbb{R}$ and any $t\in\mathbb{R}$. Let $x\in \partial X - \{l(+\infty),l(-\infty)\}$. then $d(p_l(x),p_{l'}(x))\leqslant 4\delta$.
\end{proposition}
\begin{proof}
    When $n$ is large enough, we have
    \begin{align*}
        (l(n)\cdot x)_{l(0)} - (l(-n)\cdot x)_{l(0)} & = \dfrac{1}{2} \big( d(l(n),l(0)) + d(x,l(0)) - d (l(n), x) \\ & - d(l(-n),l(0)) - d(x,l(0)) + d (l(-n), x) \big) \\
        & = \dfrac{1}{2} \big( d(l(n),l(0))  - d (l(n), x)  - d(l(-n),l(0))  + d (l(-n), x) \big) \\
        & = \dfrac{1}{2} \big( d(l(n),l(k))  - d (l(n), x)  - d(l(-n),l(k))  + d (l(-n), x) +2k \big) \\
        & = (l(n)\cdot x)_{l(k)} - (l(-n)\cdot x)_{l(k)} + k.
    \end{align*}
    Then by Lemma \ref{coarselyproj}, when $n\rightarrow +\infty$, the above equation implies
    $ | t_l(x) - \big( t_{l'}(x)+k\big) | \leqslant 4\delta$.
    Then $$d(p_l(x),p_{l'}(x)) = d( l(t_l(x)), l(t_{l'}(x)+k) ) \leqslant | t_l(x) - \big( t_{l'}(x)+k\big) |
    \leqslant 4\delta.$$
\end{proof}

\begin{lemma}\label{newqgfromftp}
    For any constants $\lambda\geqslant 1$ and $\epsilon\geqslant 0$, there exist constants $\lambda'\geqslant 1$ and $\epsilon'\geqslant 0$, such that for any geodesic
    $l:\mathbb{R}\rightarrow X$  and any $(\lambda,\epsilon)$-quasi-geodesic ray $l':\mathbb{R}_{\geqslant 0}\rightarrow X$ with $l(0)=l'(0)$ and $l(+\infty)=l'(+\infty)$, the curve
    \begin{equation*}
        L(t)=\left\{  
            \begin{aligned} l(t) & & t<0 \\
        l'(t) & & t\geqslant 0
     \end{aligned} \right.
    \end{equation*}
is a $(\lambda',\epsilon')$-quasi-geodesic.
\end{lemma}
\begin{proof}
    Let $t_1,t_2>0$ we estimate $d(l(-t_1),l'(t_2))$. By Proposition \ref{FTP}, there exists a constant $R_0$ which only depends on $\lambda,\epsilon$ and $\delta$, such that the Hausdorff distance between $l([0,+\infty)$ and $l'([0,+\infty)$ is bounded by $R_0$. Then there exists $s\geqslant 0$, such that $d(l(s),l'(t_2))\leqslant R_0$. Hence $$R_0\geqslant d(l(s),l'(t_2)) \geqslant |d(l(s),l(0))-d(l'(t_2),l'(0))|= |s -d(l'(t_2),l'(0))|.$$ Since $\lambda^{-1}t_2-\epsilon\leqslant d(l'(t_2),l'(0))\leqslant \lambda t_2+\epsilon$, we have
    \begin{align*}
        d(l(-t_1),l'(t_2)) & \geqslant d(l(-t_1),l(s))-d(l(s),l'(t_2)) \geqslant t_1+s -R_0 \\  & \geqslant t_1-R_0+\lambda^{-1}t_2-\epsilon-R_0 \geqslant \lambda^{-1}(t_1+t_2)-\epsilon-2R_0,
    \end{align*}
    and $$d(l(-t_1),l'(t_2))\leqslant d(l(-t_1),l(0))+d(l'(0),l'(t_2)) \leqslant t_1+\lambda t_2+\epsilon\leqslant \lambda(t_1+t_2)+\epsilon.$$ Hence $L$ is a $(\lambda,\epsilon+2R_0)$-quasi-geodesic.\\
\end{proof}

Now we focus on the Cayley graphs of hyperbolic groups. Let $\Gamma$ be a hyperbolic group with a finite, symmetric generating set $S$. Then the Cayley graph $\Cay(\Gamma,S)$ is $\delta$-hyperbolic for some constant $\delta\geqslant 0$. Let $d$ denote the word metric. By Proposition \ref{visualmetric}, we fix the base point $\id\in \Cay(\Gamma, S)$ and a small enough constant $\kappa$ such that $d_\kappa$ is a metric on $\partial \Gamma$ and $$\nu e^{-\kappa(x\cdot y)_{\id}} \leqslant d_\kappa (x,y)\leqslant e^{-\kappa (x\cdot y)_{\id}},$$ for any $x,y\in\partial \Gamma$, where $0< \nu =3-2e^{2\delta\kappa} \leqslant 1$.\\

For a subset $M$ of $\Cay(\Gamma, S)$, we denote $V_{r}(M)$ the closed neighborhood of $M$ of radius $r$.

\begin{lemma}\label{gofaraway}
    There exists a constant $C_1>0$ with the following property. Let $l:\mathbb{R}\rightarrow \Cay(\Gamma,S)$ be an arbitrary geodesic with $l(0)=\id$. Let $x\in\partial \Gamma$ be an arbitrary point on the boundary with $d(p_l(x),\id)\leqslant 4\delta$ and a geodesic ray
    $l':\mathbb{R}_{\geqslant 0}\rightarrow \Cay(\Gamma, S)$ with $l'(0)=p_l(x)$ and $l'(+\infty)=x$. Then we have 
    $$l'([C_1,+\infty))\cap V_{2\delta}(l(\mathbb{R}))=\emptyset.$$
\end{lemma}
\begin{proof}
    We argue by contradiction. Let $l_i:\mathbb{R}\rightarrow \Cay(\Gamma,S)$ be a sequence of geodesics with $l_i(0)=\id$. Let $x_i$ be a sequence of point on $\partial \Gamma$ with
    $d(p_{l_i}(x),\id)\leqslant 4\delta$ and $l'_i:\mathbb{R}_{\geqslant 0}\rightarrow \Cay(\Gamma, S)$ a sequence of geodesic rays with $l'_i(0)= p_l(x_i)$ and $l'_i(+\infty)= x_i$, such that
    we can find real numbers $s_i\geqslant i$ with $l'_i(s_i)\in V_{2\delta}(l(\mathbb{R}))$. Let $t_i\in\mathbb{R}$ such that $d(l'_i(s_i), l_i(t_i))\leqslant 2\delta$. We may assume $t_i\geqslant t_{l_i}(x_i)$. We consider the geodesic triangle with edges $l_i([t_{l_i}(x_i),t_i])$, $l'_i([0,s_i])$ and a geodesic segment joining $l'_i(s_i)$ and $l(t_i)$, which is $\delta$-slim. Then $l'_i([0,i])\subset l'_i([0,s_i]) \subset V_{3\delta}(l_i(\mathbb{R}))$. Moreover, since $l'_i([0,t])\subset V_{t+4\delta}(\{\id\})$ for any $t\geqslant 0$, we have $l'_i([0,t])\subset V_{t+4\delta}(\{\id\})\cap V_{3\delta}(l_i(\mathbb{R})) \subset V_{3\delta}(l_i[-t-7\delta,t+7\delta]).$ for any $0\leqslant t\leqslant i$.\\

    Since each $l_i$ passes through $\id$, we can find a constant $D_1>0$ such that $d_\kappa (l_i(+\infty),l_i(-\infty))\geqslant D_1$. Then $d(p_{l_i}(x_i),\id)\leqslant 4\delta$ implies that $$|t_{l_i}(x_i)|=|(l_i(+\infty)\cdot x_i)_{\id}-(l_i(-\infty)\cdot x_i)_{\id}|\leqslant 4\delta.$$ Then by Proposition \ref{visualmetric}, we have
    \begin{align*}
        D_1 & \leqslant d_\kappa (l_i(+\infty),l_i(-\infty)) \leqslant  d_\kappa (l_i(+\infty),x_i) + d_\kappa (x_i,l_i(-\infty)) \\
        & \leqslant e^{-\kappa(l_i(+\infty)\cdot x_i)_{\id}}+e^{-\kappa(l_i(-\infty)\cdot x_i)_{\id}} \\
        & \leqslant e^{-\kappa(l_i(+\infty)\cdot x_i)_{\id}}+e^{-\kappa(l_i(+\infty)\cdot x_i)_{\id}+4\kappa\delta} \\
        & =(1+e^{4\kappa\delta})e^{-\kappa(l_i(+\infty)\cdot x_i)_{\id}}\\
        & \leqslant \dfrac{1+e^{4\kappa\delta}}{\nu} d_\kappa (l_i(+\infty),x_i).
    \end{align*}
    For convenience, we denote $D_2= \dfrac{\nu D_1}{1+e^{4\kappa\delta}}$, then $d_\kappa (l_i(+\infty),x_i)\geqslant D_2$. Similarly, $d_\kappa (l_i(-\infty),x_i)\geqslant D_2$.\\

    We pick a subsequence $k_i$ of $\mathbb{N}$ such that $l_i$ converges to a geodesic $l$ and $l_i$ converges to a geodesic ray $l'$ with $l'(+\infty)=x\in \partial\Gamma$. Then $d_\kappa(x,l(+\infty))\geqslant D_2$, $d_\kappa(x,l(-\infty))\geqslant D_2$, hence $x\ne l(+\infty)$ and $x\ne l(-\infty)$.\\

    For any $N>0$, there exists $I>0$, such that if $i\geqslant I$, $V_{3\delta}(l_i[-N-7\delta,N+7\delta])\subset V_{3\delta+1}(l(\mathbb{R}))$ as $l_i$ converges to $l$ uniformly on any compact set. Then we have $l'([0,N])\subset V_{3\delta+2}(l(\mathbb{R}))$ since $l'_i([0,N])$ uniformly converges to $l'([0,N])$. Since $N$ is arbitrary, we have $l'([0,+\infty))\subset V_{3\delta+2}(l(\mathbb{R}))$, which contradicts $x\ne l(+\infty)$ and $x\ne l(-\infty)$.\\
\end{proof}

\begin{proposition}\label{newqgfromproj}
    There exists a constant $C_2>0$ with the following property. Let $l:\mathbb{R}\rightarrow \Cay(\Gamma,S)$ be a geodesic and $x\in\partial \Gamma$ a point on the boundary. Let $l':\mathbb{R}_{\geqslant 0}\rightarrow \Cay(\Gamma, S)$ be a geodesic ray with $l'(0)=p_l(x)$ and $l'(+\infty)=x$. Let $L$ be the curve consisting of the following two parts,
    \begin{itemize}
    \item the geodesic ray $l'$;
    \item the geodesic ray from $p_l(x)$ to $l(+\infty)$ along $l$.
    \end{itemize}
    Then $L$ is a $(1, C_2)$-quasi-geodesic parametrized by arc length.
\end{proposition}
\begin{proof}
    Up to an isometry, we may assume $l(0)=\id$ and $d(p_l(x), \id)=t_l(x)\leqslant 4\delta$ by Proposition \ref{coarselyprojrmk}. Let $t \geqslant t_l(x)$ and $s \geqslant 0$ be two arbitrary parameters. We estimate the distance between $l(t)$ and $l'(s)$.\\

    Let $g:[a,b]\rightarrow \Cay(\Gamma,S)$ be a geodesic segment with $g(a)=l(t)$ and $g(b)=l'(s)$. Since the triangle with vertices $p_l(x)$, $l(t)$ and $l'(s)$ is $\delta$-slim, there exists $c\in[a,b]$, $t_0\in [t_l(x), t]$ and $s_0\in [0,s]$, such that $$d(g(c),l(t_0))\leqslant \delta\ \text{and}\ d(g(c),l'(s_0))\leqslant \delta.$$ Then $d(l(t_0),l'(s_0))\leqslant 2\delta$. By Lemma \ref{gofaraway}, we can find $C_1$ independent to $L$, $x$ and $l'$, such that $l'([C_1,+\infty))\cap V_{2\delta}(l(\mathbb{R}))=\emptyset$. Hence we have $s_0 \leqslant C_1$ and then $t_0\leqslant C_1+2\delta +t_l(x)\leqslant C_1 +6\delta$.\\

    Therefore, we have
    \begin{align*}
        d(l(t),l'(s)) & =d(g(a),g(c))+d(g(b),g(c))=d(l(t),g(c))+d(l'(s),g(c)) \\
        & \geqslant d(l(t),l(t_0))-d(l(t_0),g(c))+d(l'(s),l'(s_0))-d(l'(s_0),g(c))\\
        & \geqslant t-t_0 -\delta + s-s_0 -\delta\\
        & \geqslant t- C_1-6\delta -\delta + s- C_1-\delta\\
        & = t+s-2C_1-8\delta. 
    \end{align*}    
    On the other hand, it is clear that $d(l(t),l'(s)) \leqslant t+s +4\delta$. The difference between the parameters of $l(t)$ and $l'(s)$ on $L$ is $t+s-t_l(x)$. Let $C_2=2C_1+12\delta$. Then we have $L$ is a $(1,C_2)$-quasi-geodesic.\\
\end{proof}\ \\

\subsection{Hölder continuity}\ \\

Now we are ready to prove that $\xi^k$ is Hölder continuous on each $Q^{+\infty}_{P,(\lambda,\epsilon,b)}$.\\

We still assume that $\Gamma$ is a hyperbolic group with a finite, symmetric generating set $S$. The Cayley graph $\Cay(\Gamma,S)$ is $\delta$-hyperbolic for a constant $\delta\geqslant 0$. Let $d$ and $|\cdot|$ denote the word metric on $\Gamma$. Take $\id\in \Cay(\Gamma, S)$ as the base point. By Proposition \ref{visualmetric}, there exist constants $\kappa >0$ and $0<\nu\leqslant 1$ such that $d_\kappa$ is a metric on $\partial \Gamma$ with $$\nu e^{-\kappa(x\cdot y)_{\id}} \leqslant d_\kappa (x,y)\leqslant e^{-\kappa (x\cdot y)_{\id}},$$ for any $x,y\in\partial \Gamma$.

\begin{theorem}
    Let $P$ be a closed, $\Gamma$-invariant subset of $\partial^{(2)} \Gamma$. Let $\rho: \Gamma\rightarrow \mathrm{GL}(d,\mathbb{R})$ be a $k$-Anosov representation over $S_P$. Then for any constant $b\geqslant 0$, the limit map $\xi^k$ is Hölder continuous on $Q^{+\infty}_{P,(1,0,b)}$.
\end{theorem}
\begin{proof}
    There exists a constant $D_1>0$, such that for any $l\in Q_{P,(1,0,b)}$, we have $d_\kappa(l(+\infty),l(-\infty))\geqslant D_1$ since $l$ is a geodesic with $d(l(0),\id)\leqslant b$. Let $x,y\in Q^{+\infty}_{P,(1,0,b)}$ be two points on the boundary with $d_\kappa(x,y)\leqslant D_2$, where $D_2>0$ is a constant which will be given later. Let $l_1,l_2\in Q_{P,(1,0,b)}$ with $l_1(+\infty)=x$ and $l_2(+\infty)=y$. Then $$D_1-D_2\leqslant d_\kappa(l_1(-\infty),y)\leqslant e^{-\kappa (l_1(-\infty)\cdot y)_{\id}}$$ and $$\nu  e^{-\kappa(x\cdot y)_{\id}} \leqslant d_\kappa (x,y)\leqslant D_2.$$ Together with the definition of Gromov product, Proposition \ref{visualmetric}, and $d(l_1(0),\id)\leqslant b$, we have $$(l_1(-\infty)\cdot y)_{l_1(0)}\leqslant (l_1(-\infty)\cdot y)_{\id}+b+2\delta\leqslant \dfrac{1}{\kappa}\log\dfrac{1}{D_1-D_2}+b+2\delta= C^-(D_2)$$ and $$(x\cdot y)_{l_1(0)}\geqslant (x\cdot y)_{\id}-b-2\delta\geqslant \dfrac{1}{\kappa}\log\dfrac{\nu}{D_2}-b-2\delta=C^+(D_2).$$ When $D_2$ goes to $0$, $C^-(D_2)$ is bounded while $C^+(D_2)$ goes to $+\infty$. We choose $D_2$ small enough such that $2C^-(D_2)+1\leqslant C^+(D_2)$. Then $t_{l_1}(y)= (x\cdot y)_{l_1(0)} -(l_1(-\infty)\cdot y)_{l_1(0)}$ is positive.\\

    We now construct a curve $l'_2$ in $\Cay(\Gamma,S)$ that consists of the following parts.
    \begin{itemize}
        \item A geodesic ray from $p_{l_1}(y)$ to $y$;
        \item The geodesic segment from $l_1(0)$ to $p_{l_1}(y)$ along $l_1$;
        \item A geodesic segment from $l_2(0)$ to $l_1(0)$;
        \item The geodesic ray from $l_2(0)$ to $l_2(-\infty)$ along $l_2$.
    \end{itemize}
    Parametrize $l'_2$ by arc length so that $l'_2(0)=l_1(0)$, $l'_2(+\infty)= y$ and $l'_2(-\infty)=l_2(-\infty)$.\\

    Then there exist constants $\lambda\geqslant 1$ and $\epsilon \geqslant 0$, which are independent to $x$ and $y$, such that $l'_2\in Q_{P,(\lambda,\epsilon,b)}$. Actually, the first two parts consist a $(1,C)$-quasi-geodesic ray by Proposition \ref{newqgfromproj}, for some $C\geqslant 0$. Then the first three parts consist a $(1,C+4b)$-quasi-geodesic ray, which has bounded Hausdorff distance (provided by Proposition \ref{FTP}) away from the geodesic ray from $l_2(0)$ to $y$ along $l_2$. Hence by Lemma \ref{newqgfromftp}, $l'_2$ is a $(\lambda,\epsilon)$-quasi-geodesic for some constants $\lambda\geqslant 1$ and $\epsilon\geqslant 0$.\\

    Now we have $l_1,l'_2\in Q_{P,(\lambda,\epsilon,b)}$ parametrized by arc length. Since $\rho$ is $k$-dominated on $\Gamma^+_{P,(\lambda,\epsilon,b)}$, by Observation \ref{observationext}, there exist constants $C_0,c_0 \geqslant 0$ such that $$\dfrac{\sigma_{k+1}(\rho(\gamma))}{\sigma_{k}(\rho(\gamma))}\leqslant C_0e^{-c_0|\gamma|},\ \text{for any}\ \gamma\in \Gamma^+_{P,(\lambda,\epsilon,b)}.$$ By Lemma \ref{estsingularvalue} (3) and Proposition \ref{uniformconvergenceoflimitmaps}, $\xi^k(x)= \lim\limits_{n\rightarrow +\infty} U_k(\rho(l_1(n)))$ where $U_k(\rho(l_1(n)))$ is a Cauchy sequence such that
    \begin{align*}
        & d(U_k(\rho(l_1(n)),U_k(\rho(l_1(n+1)))\\
        \leqslant & \Vert \rho(l_1(n))^{-1}\rho(l_1(n+1))\Vert \cdot\Vert \rho(l_1(n+1))^{-1}\rho(l_1(n))\Vert \cdot\dfrac{\sigma_{k+1}(\rho(l_1(n)))}{\sigma_k(l_1(n))}\\
        \leqslant & C_1 C_0e^{-c_0|l_1(n)|}
    \end{align*}
    where $C_1$ denotes the maximum of $\Vert \rho(s)\Vert \cdot\Vert \rho(s^{-1})\Vert $ for all generators $s\in S$. Then there exist constants $C_2,c_2>0$, that only depend on $C_0, c_0, C_1, \lambda$ and $\epsilon$, such that $$d(\xi^k(x),U_k(\rho(l_1(n))))\leqslant C_2 e^{-c_2|l_1(n)|},$$ and similarly, $$d(\xi^k(y),U_k(\rho(l'_2(n))))\leqslant C_2 e^{-c_2|l'_2(n)|},$$ for any $n\in\mathbb{N}$.\\
    
    Since $l_1=l'_2$ on $[0,t_{l_1}(y)]$, let $n=t_{l_1}(y)$ in the above two inequalities, then we have 
    \begin{align*}
        d(\xi^k(x),\xi^k(y)) & \leqslant d\big(\xi^k(x),U_k(\rho(l_1(t_{l_1}(y))))\big)+d\big(\xi^k(y),U_k(\rho(l'_2(t_{l_1}(y))))\big)\\
        & \leqslant 2C_2 e^{c_2b}\cdot e^{-c_2t_{l_1}(y)}\\
        & \leqslant  2C_2 e^{c_2b}\cdot e^{-c_2 ((x\cdot y)_{l_1(0)} -(l_1(-\infty)\cdot y)_{l_1(0)})}
    \end{align*}
    
    Recall that $(l_1(-\infty)\cdot y)_{l_1(0)}\leqslant C^-(D_2)$ and $(x\cdot y)_{l_1(0)}\geqslant (x\cdot y)_{\id}-b-2\delta$. Then $$  e^{-c_2 ((x\cdot y)_{l_1(0)} -(l_1(-\infty)\cdot y)_{l_1(0)})} \leqslant e^{c_2(C^-(D_2)+b+2\delta)}\cdot e^{-c_2 (x\cdot y)_{\id}}.$$
    
    Since $\nu  e^{-\kappa(x\cdot y)_{\id}} \leqslant d_\kappa (x,y)$, we have $e^{-c_2 (x\cdot y)_{\id}} \leqslant \left(\dfrac{1}{\nu}\right)^{\frac{c_2}{\kappa}} d_\kappa (x,y)^{\frac{c_2}{\kappa}}$. Let $C_3= 2\left(\dfrac{1}{\nu}\right)^{\frac{c_2}{\kappa}}C_2e^{c_2(C^-(D_2)+2b+2\delta)}$, then we have $$d(\xi^k(x),\xi^k(y))\leqslant C_3 d_\kappa (x,y)^{\frac{c_2}{\kappa}}.$$ The above inequality is true for any $x,y\in Q^{+\infty}_{P,(1,0,b)}$ with $d_\kappa(x,y)\leqslant D_2$. Let $C_4>0$ be a constant such that $C_4D_2^{\frac{c_2}{\kappa}}\geqslant 1$. When $d_\kappa(x,y)> D_2$, we have $$ d(\xi^k(x),\xi^k(y))\leqslant 1\leqslant C_4 d_\kappa(x,y)^{\frac{c_2}{\kappa}}.$$ Let $C_5=\max\{C_3,C_4\}$, then $$d(\xi^k(x),\xi^k(y))\leqslant C_5 d_\kappa (x,y)^{\frac{c_2}{\kappa}},$$ for any $x,y\in Q^{+\infty}_{P,(1,0,b)}$, that is, $\xi^k$ is Hölder continuous on $Q^{+\infty}_{P,(1,0,b)}$.\\
\end{proof}

Notice that for any compact subset $P_0$ of $P$, $\inf \{ d_{\kappa}(x,y)\ |\ (x,y)\in P_0\} >0$. Hence there exists a large enough constant $b\geqslant 0$, such that $P_0\subset Q^{+\infty}_{P,(1,0,b)}\times Q^{-\infty}_{P,(1,0,b)}$. By Lemma \ref{unionconverge} and considering $\hat{P}$, we have the following corollary.

\begin{corollary} 
    For any constants $\lambda\geqslant 1$, $\epsilon\geqslant 0$ and $b\geqslant 0$, $\xi^k$ is Hölder continuous on $Q^{+\infty}_{P,(\lambda,\epsilon,b)}$, and $\xi_{d-k}$ is Hölder continuous on $Q^{-\infty}_{P,(\lambda,\epsilon,b)}$. Moreover, $(\xi^k,\xi_{d-k})$ is Hölder continuous on any compact subset of $P$.
\end{corollary}\ \\

\section{Stability of Anosov representations over closed subflows}\label{secstability}
We still assume that $\Gamma$ is a hyperbolic group and $P$ is a closed, $\Gamma$-invariant subset of $\partial^{(2)}\Gamma$. In this section, we show that the set of $k$-Anosov representations over $S_P$ is open in $\mathrm{Hom}(\Gamma,\mathrm{GL}(d,\mathbb{R}))$, i.e., $k$-Anosov representations over $S_P$ are stable.\\

We recall some notations from Section \ref{equivalence1}. For a representation $\rho:\Gamma \rightarrow \mathrm{GL}(d,\mathbb{R})$, $E_\rho=\widetilde{E}/\Gamma=U\Gamma\times_{\rho}\mathbb{R}^d$. We will also denote $\widetilde{E}$ with the $\Gamma$ action induced by $\rho$ as $\widetilde{E}_\rho$. $\Vert  \cdot\Vert  $ is a Riemannian metric on $E_\rho$ that lifts to a $\Gamma$-invariant, $\mathbb{Z}/2\mathbb{Z}$-invariant metric on $\widetilde{E}$, which is still denoted by $\Vert  \cdot\Vert $.

\begin{theorem}\label{opennbhdstab}
    The set of $k$-Anosov representations over $S_P$ is open in $\mathrm{Hom}(\Gamma,\mathrm{GL}(d,\mathbb{R}))$.\\
\end{theorem}

Let $\rho_0:\Gamma \rightarrow \mathrm{GL}(d,\mathbb{R})$ be a $k$-Anosov representation over $S_P$. Our goal is to find a neighborhood of $\rho_0$ that consists of $k$-Anosov representations over $S_P$ in $\mathrm{Hom}(\Gamma,\mathrm{GL}(d,\mathbb{R}))$. Let $O$ be a neighborhood of $\rho_0$ in $\mathrm{Hom}(\Gamma, \mathrm{GL}(d,\mathbb{R}))$. We denote $\widetilde{E}_O=O\times \widetilde{U\Gamma} \times \mathbb{R}^d$ and $E_O=\widetilde{E}_O/\Gamma$ where the $\Gamma$-action on $\widetilde{E}_O$ is defined by $$\gamma (\rho,z,v)=(\rho,\gamma z,\rho(\gamma)v),$$ for any $\rho\in O$, $\gamma\in\Gamma$, $z\in \widetilde{U\Gamma}$ and $v\in \mathbb{R}^d$. The flow $\widetilde{\phi}^t$ on $O\times \widetilde{U\Gamma}$  (respectively, $\widetilde{\psi}^t$ on $\widetilde{E}_O$) is defined by $\widetilde{\phi}^t(\rho,z)= (\rho,\widetilde{\phi}^t(z))$ (respectively, $\widetilde{\psi}^t(\rho, z, v)=(\rho, \widetilde{\psi}^t(z,v))=(\rho, \widetilde{\phi}^t(z),v)$) for any $\rho\in O$, $z\in \widetilde{U\Gamma}$ and $v\in \widetilde{E}_{O,(\rho,z)}$.\\

Intuitively, when two representations $\rho_1$ and $\rho_2\in \mathrm{Hom}(\Gamma, \mathrm{GL}(d,\mathbb{R}))$ are close enough, $E_{\rho_1}$ and $E_{\rho_2}$ should be isomorphic. 
More formally, we have the following lemma.

\begin{lemma}\label{deformationvb}
    By replacing $O$ by a smaller neighborhood of $\rho_0$, there exists an isomorphism of vector bundles fibring over $id_{O\times \widetilde{U\Gamma}}$,
    $$ g: O\times \widetilde{E}_{\rho_0} \simeq \widetilde{E}_O,$$ such that,
    \begin{itemize}
        \item[(1).] $g$ is $\Gamma$-equivariant, 
        where the $\Gamma$-action on $O\times \widetilde{E}_{\rho_0}$ is defined by 
        $$\gamma (\rho,z,v)=(\rho,\gamma z,\rho_0(\gamma)v),$$
        for any $\rho\in O$, $\gamma\in\Gamma$, $z\in \widetilde{U\Gamma}$ and $v\in \mathbb{R}^d$.
        \item[(2).] $g|_{\{\rho_0\}\times \widetilde{E}_{\rho_0} }$ is the identity map on $\{\rho_0\}\times \widetilde{E}_{\rho_0}$.
    \end{itemize}
\end{lemma}
\begin{proof}
    Since $O\times \widetilde{E}_{\rho_0}$ and $\widetilde{E}_O$ are both trivial $\mathbb{R}$-vector bundles over $O\times \widetilde{U\Gamma}$, we may view the desired $g$ as a continuous map $g: O\times \widetilde{U\Gamma}\rightarrow \mathrm{GL}(d,\mathbb{R})$, with the $\Gamma$-equivariant condition, i.e., $$g(\rho,\gamma z)=\rho(\gamma)\circ g(\rho,z)\circ \rho_0(\gamma)^{-1}$$ for any $\rho\in O$, $z\in \widetilde{U\Gamma}$, $v\in\mathbb{R}^d$ and $\gamma\in\Gamma$, and $g|_{\{\rho_0\}\times \widetilde{U\Gamma}}\equiv I$.\\

    Let $d$ be the metric on $\widetilde{U\Gamma}$ given by Theorem \ref{metricunittangentbundle}, such that $\Gamma$ acts on $(\widetilde{U\Gamma},d)$ properly discontinuously and cocompactly by isometries. Let $p\in \widetilde{U\Gamma}$ be a fixed base point. Then the orbit map $\tau_p: \Gamma\rightarrow \widetilde{U\Gamma}$, $\tau_p(\gamma)=\gamma p$ is a $(\lambda,\epsilon)$-quasi-isometry for some constants $\lambda \geqslant 1$ and $\epsilon\geqslant 0$, by Theorem 60(c) \cite{Mineyev}.\\

    We consider the Dirichlet fundamental domain $$M=\{\ z\in \widetilde{{U\Gamma}}\ |\ d(z,p)\leqslant d(z,\gamma p),\forall \gamma\in\Gamma\ \},$$ which is bounded since $\tau_p$ is a quasi-isometry. Then there exists a constant $C>0$, such that $$M \subset B(p,C)=\{z\in \widetilde{U\Gamma}\ |\ d(z,p)< C \}.$$ For any $\alpha\in \Gamma$ with $|\alpha|> \lambda(2C+\epsilon)$ and any $z\in B(p,C)$, we have $$d(\alpha z, p)\geqslant d(\alpha p, p)-d(\alpha z, \alpha p) = d(\alpha p, p)-d( z,  p) > \lambda^{-1}|\alpha|-\epsilon-C> C,$$ hence $\alpha z\not\in B(p,C)$. We denote $A=\{\alpha\in \Gamma\ |\ |\alpha|>\lambda(2C+\epsilon) \}$. Then $B=\Gamma-A$ is finite. We denote the cardinal of $B$ by $|B|$.\\

    Let $U$ be a small enough open ball centered at $I$ in $\mathrm{GL}(d,\mathbb{R})$, such that any summation of at most $|B|$ many elements in $U$ is still invertible. We shrink $O$ such that $\rho(\gamma)\rho_0(\gamma)^{-1}\subset U$ for any $\rho\in O$ and $\gamma\in B$.\\

    Since $M$ is properly contained in $B(p,C)$, let $f_0$ be a continuous real function on $O\times \widetilde{U\Gamma}$ supported in $O\times B(p,C)$, such that $0\leqslant f_0 \leqslant 1$ and $f_0|_{O\times M}\equiv 1$. We define the map $g_0: O\times \widetilde{U\Gamma} \rightarrow \mathrm{Mat}_{d \times d}(\mathbb{R}) $ by $g_0= f_0 I$. Let $g: O\times \widetilde{U\Gamma} \rightarrow \mathrm{Mat}_{d \times d}(\mathbb{R}) $ be the map defined by $$g(\rho,z)=\dfrac{1}{\sum_{\gamma\in\Gamma} f_0(\rho,\gamma^{-1}z)} \sum_{\gamma\in\Gamma} \rho(\gamma)\circ g_0(\rho,\gamma^{-1}z)\circ \rho_0(\gamma)^{-1}.$$

    We now show that $g$ is the desired map. Firstly, $g$ is well-defined since there are at most $|B|$ nonzero terms in the summation above at any point and $\sum_{\gamma\in\Gamma} f_0(\rho,\gamma^{-1}z)$ is positive everywhere. Indeed, when $z\in M$, we only need to take the summation for $\gamma\in B$, since $f_0$ and $g_0$ are supported in $O\times B(p,C)$ but $\gamma z \not\in B(p,C)$ when $\gamma\in A=\Gamma-B$. Similarly, when $z\in \eta M$ for some $\eta\in \Gamma$, we only need to take the summation for $\gamma\in B\eta$. The denominator $\sum_{\gamma\in\Gamma} f_0(\rho,\gamma^{-1}z)$ is positive everywhere since when $z\in \eta M$ for some $\eta\in \Gamma$, $f_0(\rho, \eta^{-1}z)=1$ and the other terms are all non-negative.\\

    Secondly, $g$ is $\rho$-equivariant since 
    \begin{align*}
        g(\rho,\eta z) &= \dfrac{1}{\sum_{\gamma\in\Gamma} f_0(\rho,\gamma^{-1}\eta z)} \sum_{\gamma\in\Gamma} \rho(\gamma)\circ g_0(\rho,\gamma^{-1}\eta z)\circ \rho_0(\gamma)^{-1}\\
        &= \dfrac{1}{\sum_{\gamma\in\Gamma} f_0(\rho,(\eta^{-1}\gamma)^{-1} z)} \sum_{\gamma\in\Gamma} \rho(\eta)\rho(\eta^{-1}\gamma)\circ g_0(\rho,(\eta^{-1}\gamma)^{-1} z)
        \circ \rho_0(\eta^{-1}\gamma)^{-1}\rho_0(\eta)^{-1}\\
        &= \rho(\eta)\Bigg( \dfrac{1}{\sum_{\eta\gamma\in\Gamma} f_0(\rho,\gamma^{-1}z)} \sum_{\eta\gamma\in\Gamma} \rho(\gamma)\circ g_0(\rho,\gamma^{-1} z)\circ \rho_0(\gamma)^{-1}\Bigg)\rho_0(\eta)^{-1} \\
        &= \rho(\eta)\Bigg( \dfrac{1}{\sum_{\gamma\in\Gamma} f_0(\rho,\gamma^{-1}z)} \sum_{\gamma\in\Gamma} \rho(\gamma)\circ g_0(\rho,\gamma^{-1} z)\circ \rho_0(\gamma)^{-1}\Bigg)\rho_0(\eta)^{-1} \\
        &= \rho(\eta)g(\rho,z)\rho_0(\eta)^{-1}.\\
    \end{align*}

    By our choice of $U$ and $O$, $g(\rho,z)$ is invertible for any $(\rho,z)\in O\times M$. Then the value $g$ is invertible at any point of $O\times \widetilde{U\Gamma}$
    by the $\rho$-equivariance condition, and hence
    $g:O\times \widetilde{U\Gamma}\rightarrow  \mathrm{Mat}_{d \times d}(\mathbb{R})$ has image in $\mathrm{GL}(d,\mathbb{R})$.\\

    Finally, for any $z\in \widetilde{U\Gamma}$,
    \begin{align*}
        g(\rho_0,z)&=\dfrac{1}{\sum_{\gamma\in\Gamma} f_0(\rho_0,\gamma^{-1}z)} \sum_{\gamma\in\Gamma} \rho_0(\gamma)\circ f_0(\rho_0,\gamma^{-1}z)I\circ \rho_0(\gamma)^{-1}\\
        &= \dfrac{1}{\sum_{\gamma\in\Gamma} f_0(\rho_0,\gamma^{-1}z)} \sum_{\gamma\in\Gamma}  f_0(\rho_0,\gamma^{-1}z)I = I
    \end{align*}
    Hence $g|_{\{\rho_0\}\times \widetilde{U\Gamma}}\equiv I$.\\
\end{proof}

Since $\rho_0$ is $k$-Anosov over $S_P$, by Theorem \ref{equivdands}, $E_{\rho_0}$ admits a $k$-dominated splitting over $S_P$, denoted by $E_{\rho_0}|_{S_P}=E^s_{\rho_0}\oplus E^u_{\rho_0}$, where $E^s_{\rho_0}$ and $E^u_{\rho_0}$ are $\psi$-invariant subbundles of $E_{\rho_0}|_{S_P}$ over $S_P$, with $\mathrm{rank}(E^s_{\rho_0})=k$, and there are constants $C,\lambda>0$, such that $$\frac{\Vert \psi_q^t(v)\Vert }{\Vert \psi_q^t(w)\Vert }\leqslant Ce^{-\lambda t}\dfrac{\Vert v\Vert }{\Vert w\Vert },$$ for any $q\in S_P, t\in \mathbb{R}_{+}$ and nonzero vectors $v\in E^s_{\rho_0,q}$, $w\in E^u_{\rho_0,q}$.\\

Let $g: O\times \widetilde{E}_{\rho_0} \simeq \widetilde{E}_O$ be an isomorphism given by Lemma \ref{deformationvb}. This descends to an isomorphism (still denoted by $g$) $g: O\times {E}_{\rho_0} \simeq {E}_O$. Then ${E}_O|_{O\times S_P} = g(O\times {E}^s_{\rho_0}) \oplus g(O\times {E}^u_{\rho_0})$. We denote ${E}_O^s=g(O\times {E}^s_{\rho_0})$ and ${E}_O^u=g(O\times {E}^u_{\rho_0})$, which are subbundle of ${E}_O|_{O\times S_P}$ over $O\times S_P$. In particular, ${E}_O^s|_{\{\rho_0\}\times S_P} = E^s_{\rho_0}$ and ${E}_O^u|_{\{\rho_0\}\times S_P} = E^u_{\rho_0}$. Recall that we have a Riemannian metric $\Vert \cdot\Vert $ on $E_{\rho_0}$, which extends to a metric on $O\times E_{\rho_0}$ and then induces a metric on the vector bundle $E_O$ through $g$, still denoted by $\Vert \cdot\Vert $.\\

Since $g$ is continuous, $\psi$ is continuous on $E_O|_{O\times S_P}$. However, the splitting ${E}_O|_{O\times S_P} = {E}_O^s\oplus {E}_O^u$ may not be $\psi^t$-invariant. We use a standard method (see for example, the proof of \cite{CZZ} Theorem 8.1 (1)) to find a $\psi^t$-invariant splitting of ${E}_O|_{O\times S_P}$ using the $k$-Anosov property.\\

We write $\psi^t= \left(\begin{matrix}
    \psi^t_{(1,1)} & \psi^t_{(1,2)}\\
    \psi^t_{(2,1)} & \psi^t_{(2,2)}
\end{matrix}\right)$  with linear maps
$\psi^t_{(1,1)}: {E}_O^s \rightarrow {E}_O^s$,
$\psi^t_{(1,2)}: {E}_O^u \rightarrow {E}_O^s$,
$\psi^t_{(2,1)}: {E}_O^s \rightarrow {E}_O^u$ and 
$\psi^t_{(2,2)}: {E}_O^u \rightarrow {E}_O^u$.\\

For a linear map $h$ between two subbundles of $E_O$, we denote $\Vert  h \Vert_{(\rho,q)}$ the operator norm of $h$ at $(\rho,q)\in O\times S_P$ and $$\Vert  h \Vert_{O\times S_P}=\sup_{(\rho,q)\in O\times S_P} \Vert h \Vert_{(\rho,q)}$$ the \emph{global operator norm}.

\begin{lemma}\label{blockedflow}
    After replacing $O$ by a smaller neighborhood of $\rho_0$ again, we can find a constant $T>0$ such that for any $t\in [T,2T]$, we have
    \begin{itemize}
        \item[(1).] $\psi^t_{(1,1)}$ and $\psi^t_{(2,2)}$ are invertible and $$\sup_{(\rho,q)\in O\times S_P} \Vert \psi^t_{(1,1)}\Vert_{(\rho,q)}\cdot \Vert (\psi^t_{(2,2)})^{-1}\Vert_{(\rho,\phi^t(q))} \leqslant \dfrac{1}{3};$$
        \item[(2).] $$\Vert  (\psi^t_{(1,1)})^{-1} \psi^t_{(1,2)} \Vert_{O\times S_P} \leqslant \dfrac{1}{3};$$
        \item[(3).] $$\Vert  (\psi^t_{(2,2)})^{-1} \psi^t_{(2,1)} \Vert_{O\times S_P} \leqslant \dfrac{1}{3}.$$
    \end{itemize}
\end{lemma}
\begin{proof}
    Since $E_{\rho_0}|_{S_P}=E^s_{\rho_0}\oplus E^u_{\rho_0}$ is a $k$-dominated splitting over $S_P$, we have $\psi^t_{(1,2),(\rho_0,q)}=0$, $\psi^t_{(2,1),(\rho_0,q)}=0$,  
    $\psi^t_{(1,1),(\rho_0,q)}$ and $\psi^t_{(2,2),(\rho_0,q)}$ are invertible, and
    $$\Vert \psi^t_{(1,1)}\Vert_{(\rho_0,q)}\cdot \Vert (\psi^t_{(2,2)})^{-1}\Vert_{(\rho_0,\phi^t(q))} =
    \dfrac{\Vert \psi^t_{(1,1)}\Vert_{(\rho_0,q)}}{ m(\psi^t_{(2,2),(\rho_0,q)}) } = \dfrac{\Vert \psi^t|_{E_{\rho_0}^s}\Vert  }{m(\psi^t|_{E_{\rho_0}^u})} \leqslant Ce^{-\lambda t},$$
    for any $q\in S_P$ and $t\in \mathbb{R}_+$.
    Hence we can find a constant $T>0$, such that $$\sup_{q\in  S_P}\Vert \psi^t_{(1,1)}\Vert_{(\rho_0,q)}\cdot \Vert (\psi^t_{(2,2)})^{-1}\Vert_{(\rho_0,\phi^t(q))} \leqslant \dfrac{1}{6},$$
    for any $t\geqslant T$.
    By the continuity of $\psi^t$ and the compactness of $S_P$, we can replace $O$ by a smaller neighborhood of $\rho_0$, such that $\psi^t_{(1,1)}$ and $\psi^t_{(2,2)}$ are invertible and 
    $$\sup_{(\rho,q)\in O\times S_P} \Vert \psi^t_{(1,1)}\Vert_{(\rho,q)}\cdot \Vert (\psi^t_{(2,2)})^{-1}\Vert_{(\rho,\phi^t(q))} \leqslant \dfrac{1}{3},$$
    for any $t\in [T,2T]$.\\

    By shrinking $O$ again, we will have 
    $\Vert  (\psi^t_{(1,1)})^{-1} \psi^t_{(1,2)} \Vert_{O\times S_P} \leqslant \frac{1}{3}$ and
    $\Vert  (\psi^t_{(2,2)})^{-1} \psi^t_{(2,1)} \Vert_{O\times S_P} \leqslant \frac{1}{3}$ for any $t\in [T,2T]$,
    since $[T,2T]$, $S_P$ are compact and $\psi^t_{(1,2),(\rho_0,q)}=0$, $\psi^t_{(2,1),(\rho_0,q)}=0$ for any $q\in S_P$.\\
\end{proof}

Now we consider the vector bundle $\mathrm{Hom}({E}_O^u, {E}_O^s)$ over $O\times S_P$. For any $r>0$, let $F_r$ denote the set of sections $f$ of $\mathrm{Hom}({E}_O^u, {E}_O^s)$ with
$\Vert f\Vert_{O\times S_P}\leqslant r$, then $F_r$ is a compact metric space.

\begin{lemma}\label{contractionmap}
    For any $t\in [T,2T]$,
    define the flow $\Psi^t$ on $F_1$ by $$\Psi^t(f)= ( \psi^t_{(1,1)}f+ \psi^t_{(1,2)} )\cdot ( \psi^t_{(2,1)}f +\psi^t_{(2,2)} )^{-1},$$ for any $f\in F_1$ (or equivalently, $\Psi^t$ is defined such that $\psi^t$ maps the graph of $f_{(\rho,q)}$ to the graph of $\Psi^t(f_{(\rho,q)})$, for any $(\rho,q)\in O\times S_P$). Then $\Psi^t$ is a well-defined contraction map on $F_1$.
\end{lemma}
\begin{proof} Let $f$ be a section of $\mathrm{Hom}({E}_O^u, {E}_O^s)$ and $\Vert f\Vert_{O\times S_P}\leqslant 1$.
    By computation, we have that
    \begin{align*}
        \Psi^t(f) = & ( \psi^t_{(1,1)}f+ \psi^t_{(1,2)} )\cdot ( \psi^t_{(2,1)}f +\psi^t_{(2,2)} )^{-1}\\
        = &(\psi^t_{(1,1)})( f+ (\psi^t_{(1,1)})^{-1}\psi^t_{(1,2)} )\cdot ( (\psi^t_{(2,2)})^{-1}\psi^t_{(2,1)}f +\id )^{-1}(\psi^t_{(2,2)})^{-1},
    \end{align*}
    where $ (\psi^t_{(2,2)})^{-1}\psi^t_{(2,1)}f +\id $ is invertible since $$\Vert (\psi^t_{(2,2)})^{-1}\psi^t_{(2,1)}f\Vert_{O\times S_P}\leqslant 
    \Vert (\psi^t_{(2,2)})^{-1}\psi^t_{(2,1)}\Vert_{O\times S_P}\cdot \Vert f\Vert_{O\times S_P}\leqslant \dfrac{1}{3},$$
    by Lemma \ref{blockedflow}, and so $ \Vert ((\psi^t_{(2,2)})^{-1}\psi^t_{(2,1)}f +\id )^{-1}\Vert_{O\times S_P}\leqslant \dfrac{1}{1-\frac{1}{3}} =\dfrac{3}{2}$. Then 
    \begin{align*}
        \Vert \Psi^t(f)\Vert_{O\times S_P} &
        =\Vert (\psi^t_{(1,1)})( f+ (\psi^t_{(1,1)})^{-1}\psi^t_{(1,2)} )\cdot ( (\psi^t_{(2,2)})^{-1}\psi^t_{(2,1)}f +\id )^{-1}(\psi^t_{(2,2)})^{-1}\Vert_{O\times S_P}\\
        & \leqslant \Vert f+ (\psi^t_{(1,1)})^{-1}\psi^t_{(1,2)}\Vert_{O\times S_P}\cdot \Vert ( (\psi^t_{(2,2)})^{-1}\psi^t_{(2,1)}f +\id )^{-1}\Vert_{O\times S_P}\\
        & \cdot \sup_{(\rho,q)\in O\times S_P} \Vert \psi^t_{(1,1)}\Vert_{(\rho,q)}\cdot \Vert (\psi^t_{(2,2)})^{-1}\Vert_{(\rho,\phi^t(q))}\\
        & \leqslant ( 1+\dfrac{1}{3} )\cdot \dfrac{3}{2}\cdot \dfrac{1}{3}= \dfrac{2}{3} <1.
    \end{align*}
    Hence $\Psi$ is well-defined such that $\Psi^t(F_1)\subset F_{\frac{2}{3}}\subset F_1$.\\

    Similar to the computation above, for sections $f,g$ of $\mathrm{Hom}({E}_O^u, {E}_O^s)$ with $f\in F_1$, $$\dfrac{d}{ds}(\Psi^t(f+sg))|_{s=0} = ( \psi^t_{(1,1)}f+ \psi^t_{(1,2)} ) \cdot \dfrac{d}{ds}[( \psi^t_{(2,1)}(f+sg) +\psi^t_{(2,2)} )^{-1}]|_{s=0} + \psi^t_{(1,1)}g\cdot ( \psi^t_{(2,1)}f +\psi^t_{(2,2)} )^{-1}.$$ Since $( \psi^t_{(2,1)}(f+sg) +\psi^t_{(2,2)} )^{-1}( \psi^t_{(2,1)}(f+sg) +\psi^t_{(2,2)} )= \id $, apply $\dfrac{d}{ds}|_{s=0}$, then $$\dfrac{d}{ds}[( \psi^t_{(2,1)}(f+sg) +\psi^t_{(2,2)} )^{-1}]|_{s=0} =-( \psi^t_{(2,1)}f+ \psi^t_{(2,2)} )^{-1}\psi^t_{(2,1)}g ( \psi^t_{(2,1)}f +\psi^t_{(2,2)} )^{-1}.$$ Then
    \begin{align*}
        \dfrac{d}{ds}(\Psi^t(f+sg))|_{s=0} & =-( \psi^t_{(1,1)}f+ \psi^t_{(1,2)} )\cdot ( \psi^t_{(2,1)}f+ \psi^t_{(2,2)} )^{-1} \psi^t_{(2,1)}g ( \psi^t_{(2,1)}f +\psi^t_{(2,2)} )^{-1}\\
        &\ \  + \psi^t_{(1,1)}g\cdot ( \psi^t_{(2,1)}f +\psi^t_{(2,2)} )^{-1}\\
        & = \psi^t_{(1,1)} \Big(- ( f+ (\psi^t_{(1,1)})^{-1}\psi^t_{(1,2)} )\\ 
        &\ \  \cdot ( (\psi^t_{(2,2)})^{-1}\psi^t_{(2,1)}f +\id )^{-1}
        (\psi^t_{(2,2)})^{-1}\psi^t_{(2,1)}g ( (\psi^t_{(2,2)})^{-1}\psi^t_{(2,1)}f +\id )^{-1}\\
        &\ \ +g\cdot ( (\psi^t_{(2,2)})^{-1}\psi^t_{(2,1)}f +\id )^{-1} \Big) (\psi^t_{(2,2)})^{-1}.
    \end{align*}
    Therefore, $$ \Vert \dfrac{d}{ds}(\Psi^t(f+sg))|_{s=0}\Vert_{O\times S_P}\leqslant \dfrac{1}{3}\Big( (1+\dfrac{1}{3})\cdot\dfrac{3}{2}\cdot\dfrac{1}{3}\cdot\Vert g\Vert_{O\times S_P}\cdot\dfrac{3}{2}+ \Vert g\Vert_{O\times S_P}\cdot \dfrac{3}{2} \Big) \leqslant \dfrac{5}{6} \Vert g\Vert_{O\times S_P}.$$ Hence for two sections $f_1,f_2\in F_1$, $$\Vert \Psi^t(f_1)-\Psi^t(f_2)\Vert_{O\times S_P}\leqslant \dfrac{5}{6}\Vert f_1-f_2\Vert_{O\times S_P}.$$
\end{proof}

By the contraction mapping theorem, for each $t\in [T,2T]$, we can find $f^{(t)}\in F_1$ the unique fixed point of $\Psi^t$. Since $t\mapsto\Psi^t$ is continuous on $[T,2T]$, $f^{(t)}$ is also a continuous function on $[T,2T]$. For two rational numbers $t_1,t_2\in [T,2T]$, we can find $n_1,n_2\in \mathbb{N}_+$, such that $n_1t_1=n_2t_2$. Then $(\Psi^{t_1})^{mn_1}(f^{(t_2)})$ will converge to the fixed point of $\Psi^{t_1}$ which is $f^{(t_1)}$, as $m\rightarrow +\infty$. On the other hand, $(\Psi^{t_1})^{mn_1}(f^{(t_2)})=\Psi^{mn_1t_1}(f^{(t_2)})= \Psi^{mn_2t_2}(f^{(t_2)})= (\Psi^{t_2})^{mn_2}(f^{(t_2)})= f^{(t_2)}$. Hence $f^{(t_1)}=f^{(t_2)}$. Therefore $f^{(t)}\equiv f_0$ is a constant on $[T,2T]$.\\

The graph of $f_0$ defines a subbundle of $E_{O}|_{O\times S_P}$, denoted by $\hat{E}^u_O$, which is $\psi^t$-invariant for any $t\in\mathbb{R}$. Actually, $f_0$ is $\Psi^s$-invariant for any $s\in [T,2T]$, and for any $t\in\mathbb{R}$, we may write $t=s+nT$ for some $n\in\mathbb{Z}$. Then $\Psi^t(f_0)=\Psi^{s}(f_0)=f_0$. By the uniqueness of the fixed point, $\hat{E}^u_O|_{\{\rho_0\}\times S_P} = E^u_{\rho_0}$.\\

Similarly, there exists a unique section $g_0$ of $\mathrm{Hom}({E}_O^s, {E}_O^u)$, with the global operator norm $\Vert g_0\Vert_{O\times S_P}\leqslant 1$, which is the unique fixed point of the flow $$\Phi^t(g)=( \psi^t_{(2,2)}g+ \psi^t_{(2,1)} )\cdot ( \psi^t_{(1,2)}g +\psi^t_{(1,1)} )^{-1},$$ for any $t\in [T,2T]$ and section $g$ of $\mathrm{Hom}({E}_O^s, {E}_O^u)$ with $\Vert g\Vert_{O\times S_P}\leqslant 1$. The graph of $g_0$ is a subbundle of $E_{O}|_{O\times S_P}$, denoted by $\hat{E}^s_O$, which is $\psi^t$-invariant and $\hat{E}^s_O|_{\{\rho_0\}\times S_P} = E^s_{\rho_0}$.\\

By shrinking $O$ again, we assume $\hat{E}^s_O$ and $\hat{E}^u_O$ are transverse on $O\times S_P$ as they are continuous and transverse on $\{\rho_0\}\times S_P$, which is compact.\\

Now we rewrite $\psi^t$ with respect to the new $\psi$-invariant splitting $E_O|_{O\times S_P}=\hat{E}^s_O\oplus\hat{E}^u_O$ as
$\psi^t= \left(\begin{matrix}
    \psi^t_s & 0\\
    0 & \psi^t_u
\end{matrix}\right)$
with linear transformations $\psi^t_{s}: \hat{E}_O^s \rightarrow \hat{E}_O^s$, $\psi^t_{u}: \hat{E}_O^u \rightarrow \hat{E}_O^u$. By the same argument as in Lemma \ref{blockedflow}, after shrinking $O$ once again, $\psi^t_{s}$ and $\psi^t_{u}$ are invertible and there exists $T>0$, such that $$\sup_{(\rho,q)\in O\times S_P} \Vert \psi^t_{s}\Vert_{(\rho,q)}\cdot \Vert (\psi^t_{u})^{-1}\Vert_{(\rho,\phi^t(q))} \leqslant \dfrac{1}{3},$$ for any $t\geqslant T$. Since $S_P$ and $[0,T]$ are compact, and we may choose $O$ relatively compact, there exists a constant $C>0$ such that $$\Vert \psi^t_{s}\Vert_{(\rho,q)}\cdot \Vert (\psi^t_{u})^{-1}\Vert_{(\rho,\phi^t(q))} = \dfrac{\Vert \psi^t_{s}\Vert_{(\rho,q)}}{ m(\psi^t_{u,(\rho,q)}) } = \dfrac{\Vert \psi^t|_{\hat{E}_{O,(\rho,q)}^s}\Vert  }{m(\psi^t|_{\hat{E}_{O,(\rho,q)}^u})} \leqslant C\left(\dfrac{1}{3}\right)^{\frac{t}{T}-1}.$$ Therefore, the splitting $E_O|_{O\times S_P}=\hat{E}^s_O\oplus\hat{E}^u_O$ restricted to $\rho\times S_P$ gives a $k$-dominated splitting of $\{\rho\}\times E_\rho\simeq E_\rho$, for any $\rho\in O$. Hence by Theorem \ref{equivdands}, $\rho$ is $k$-Anosov over $S_P$, for any $\rho\in O$, which finishes the proof of Theorem \ref{opennbhdstab}.

\begin{remark}
    We actually showed that for a representation $\rho_0$ which is $k$-Anosov over $S_P$, there exists a neighborhood $O$ of $\rho_0$, such that the representations in $O$ admit $k$-dominated splitting over $S_P$ with uniform constants $C$ and $\lambda$ in Definition \ref{defdomsplit}. Then by Remark \ref{midsingular} and the proof of Theorem \ref{mainthm1}, they are $k$-Anosov over $S_P$ and $k$-dominated on $\Gamma_P^+$ with uniform constants $C$ and $\lambda$ in Definition \ref{DefAnosovoverSp} and Definition \ref{defdomprop} respectively.
\end{remark}\ \\

\section{Examples}\label{theexamples}
In this section, we introduce several examples related to our results.\\

\subsection{Primitive-stable representations}\label{exampleps} \ \\

We recall the definition of primitive-stable representations introduced by Minsky \cite{Min}.\\

Let $F_n$ be a free group of rank $n$ with a fixed generating set $S_0$ with $|S_0|=n$ and $S=S_0\cup (S_0)^{-1}$. Let $\Cay(F_n,S)$ denote the Cayley graph of $F_n$, which is a tree of degree $2n$.

\begin{definition}
    An element $\gamma$ of $F_n$ is \emph{primitive}, if there exist another $n-1$ elements $\gamma_2,\gamma_3,...,\gamma_n\in F_n$, such that $\{ \gamma,\gamma_2,...,\gamma_n\}$ generates $F_n$. Let $\mathcal{P}$ denote the set of all primitive elements of $F_n$. Equivalently, $\mathcal{P}=\mathrm{Aut}(F_n)\cdot S_0$.
\end{definition}

Let $P$ be the closure of the set $\{(\gamma^+,\gamma^-)\ | \ \gamma \in \mathcal{P} \}$ in $\partial^{(2)}F_n$. If $\gamma\in \mathcal{P}$, we call the $\gamma$-axis, i.e., the unique geodesic in $\Cay(F_n,S)$ with endpoints $(\gamma^+,\gamma^-)$, a \emph{primitive geodesic}.\\

For a representation $\rho : F_n \to \mathrm{PSL}(2,\mathbb{C})$ and a fixed point $x \in \mathbb{H}^3$, we define the $\rho$-equivariant orbit map $\tau_{x,\rho}:\Cay(F_n, S)\to\mathbb{H}^3$ by $\gamma \mapsto \rho(\gamma)\cdot x$, for any $\gamma\in F_n$, which extends linearly on the edges of $\Cay(F_n,S)$.

\begin{definition}\label{defps}
Say $\rho: F_n \to \mathrm{PSL}(2,\mathbb{C})$ is \emph{primitive-stable} if there exist constants $\lambda\geqslant 1$ and $\epsilon \geqslant 0$ such that $\tau_{x,\rho}$ maps each primitive geodesic to a $(\lambda,\epsilon)$-quasi geodesic in $\mathbb{H}^3$.
\end{definition}

The definition is independent of the choice of base point $x$.\\

The ratio of singular values of matrices in $\mathrm{PSL}(2,\mathbb{C})$ are well-defined (See Remark \ref{genLiegroup}). Notice that if we choose the base point $x$ to be the point fixed by $\mathrm{PU}(2)\subset \mathrm{PSL}(2,\mathbb{C})$, then $d(x,\rho(\gamma)\cdot x)=\log\dfrac{\sigma_1(\rho(\gamma))}{\sigma_2(\rho(\gamma))}$. Thus $\rho$ is primitive-stable if and only if there exist constants $\lambda\geqslant 1$ and $\epsilon \geqslant 0$ such that $$\log\dfrac{\sigma_{1}(\rho(\eta_i ))}{\sigma_{2}(\rho(\eta_i))} \geqslant \lambda^{-1} i - \epsilon,$$ for any primitive geodesic ray $\{ \eta_i \}_{i=0}^{\infty}$ with $\eta_0=\id$ in $F_n$. This motivates the following generalization of primitive-stable representations to higher rank Lie groups by Guéritaud--Guichard--Kassel--Wienhard \cite{GGKW}. See also Kim--Kim \cite{KK} and Kim--Tan--Zhang \cite{KTZ} for some study of this type of representations.

\begin{definition}\label{defhrps}
    A representation $\rho:F_n\rightarrow \mathrm{PGL}(d,\mathbb{R})$ is \emph{$k$-primitive-stable} if there exist constants $C,C'>0$, such that 
    $$\log\dfrac{\sigma_{d-k}(\rho(\eta_i ))}{\sigma_{d-k+1}(\rho(\eta_i))} \geqslant C i - C',$$
    for any primitive geodesic ray $\{ \eta_i \}_{i=0}^{\infty}$ with $\eta_0=\id$ in $F_n$.
\end{definition}

Notice that the collection of all elements along a primitive geodesic ray $\{ \eta_i \}_{i=0}^{\infty}$ with $\eta_0=\id$ in $F_n$ is precisely the set $(F_n)^+_P$. Thus, $\rho$ is $k$-primitive-stable if and only if it is $k$-dominated on $(F_n)^+_P$. \\

\subsection{Directed Anosov representations}\label{directedanosov}\ \\

Kim--Tan--Zhang \cite{KTZ} introduced the notion of directed Anosov representations, which they used to construct primitive-stable representations of order two free groups to higher rank Lie groups.\\

Let $\Gamma$ be a hyperbolic group, $S_0$ a finite generating set and $S=S_0\cup S_0^{-1}$. Suppose that $S_0\ne S_0^{-1}$.

\begin{definition}\label{defda}
    A representation $\rho:\Gamma \rightarrow\mathrm{PGL}_d(\mathbb{R})$ is a \emph{$(k,S_0)$-directed Anosov representation} if there exist constants $C,C'>0$ such that $$\log\dfrac{\sigma_{d-k}(\rho(\eta_i ))}{\sigma_{d-k+1}(\rho(\eta_i))} \geqslant C i - C',$$ for any geodesic ray $\{ \eta_i \}_{i=0}^{\infty}$ in $\Gamma$ with $\eta_0=\id$ and $\eta_i^{-1}\eta_{i+1}\in S_0$ for any $i\in \mathbb{N}$. 
\end{definition}

We denote $$G=\{ l:\mathbb{R}\rightarrow \Cay(\Gamma, S)\ | \ l\ \text{is a geodesic with}\ l(0)=\id,\ l(n)^{-1}l(n+1)\in S_0,\ \forall n\in\mathbb{Z}\ \}.$$ We set $P_{S_0}$ to be the closure of $\Gamma\cdot \{ (l(+\infty),l(-\infty))\in \partial^{(2)}\Gamma\ |\ l\in G \ \}$. In our context, a $(k,S_0)$-directed Anosov representation is a $k$-dominated representation on $\Gamma_{P_{S_0}}^+$.\\

Here is a simple example for this type of representations.

\begin{example}\label{simplestexample}
    Let $\rho:\mathbb{Z}\rightarrow \mathrm{GL}(3,\mathbb{R})$ be a representation defined by 
    $\rho(n)=\left(\begin{matrix}
        4^n & 0 & 0 \\
        0 & \frac{1}{2^n} & 0 \\
        0 &  0 &\frac{1}{2^n} \\
    \end{matrix}\right),$
    for any $n\in \mathbb{Z}$.
    We consider $P=\{(+\infty,-\infty)\}$. Then $\rho$ is $1$-dominated on $\mathbb{N}_+=\mathbb{Z}_P^+$, hence $\rho$ is $1$-Anosov over $S_P$, but not $2$-Anosov over $S_P$.
\end{example}\ \\

\subsection{Pleated surfaces}\label{pleatedsurfaceeg}\ \\

Let $S$ be a closed hyperbolic surface and $\pi_1(S)$ denote the fundamental group of $S$. We call a subset $\lambda\subset S$ a \emph{(geodesic) lamination}, if $\lambda$ is a collection of disjoint simple geodesics in $S$ and whose union is a closed subset of $S$. We say a lamination $\lambda$ is \emph{maximal} if it is not properly contained in a larger lamination. It is not hard to see that $\lambda$ is maximal if and only if $S-\lambda$ is a union of ideal triangles.\\

Let $\lambda$ be a maximal geodesic lamination on $S$, and let $\widetilde{\lambda}$ be the lift of $\lambda$ to the universal covering $\widetilde{S}$ of $S$. The closure of a component of $\widetilde{S}-\widetilde{\lambda}$ is called a \emph{plaque}, which is an ideal triangle. Let $\widetilde{\Delta}_{\widetilde{\lambda}}$ denote the set of plaques. Let $P_\lambda$ denote the subset of $\partial^{(2)}\pi_1(S)$ consisting of the endpoints of geodesics in $\widetilde{\lambda}$, i.e., if $l$ is a (unoriented) geodesic in $\lambda$ with endpoints $x$ and $y$, then $(x,y),(y,x)\in P_\lambda$.

\begin{definition}[\cite{Bonahon} Section 7]\label{defpls}
    A pair $(\widetilde{f},\rho)$ is called a \emph{pleated surface with topological type $S$ and pleating locus $\lambda$}, if $\rho: \pi_1(S)\rightarrow \mathrm{PSL}(2,\mathbb{C})$ is a representation and $\widetilde{f}:\widetilde{S}\rightarrow \mathbb{H}^3$ is a $\rho$-equivariant, continuous map, such that the path metric $m$ induced by the pull back of the metric on $\mathbb{H}^3$ is a hyperbolic metric, $\widetilde{f}$ maps each geodesic in $\widetilde{\lambda}$ to a geodesic in $\mathbb{H}^3$ and $\widetilde{f}$ is totally geodesic (that is, $\widetilde{f}$ maps each geodesic segment to a geodesic segment)
    on $\widetilde{S}-\widetilde{\lambda}$.
\end{definition}

The pull back metric $m$ on $S$ may not be the same as the original hyperbolic metric. $\widetilde{f}$ may not be totally geodesic on $\widetilde{S}$ since it might be bent along the geodesics in $\widetilde{\lambda}$. Such bending can be described by \emph{transverse cocycle}.

\begin{definition}\label{transversecocycle}
    A transverse cocycle for $\lambda$ is a map $\alpha: \widetilde{\Delta}_{\widetilde{\lambda}} \times \widetilde{\Delta}_{\widetilde{\lambda}} \rightarrow \mathbb{R}/2\pi \mathbb{Z}$ with the following properties
\begin{itemize}
    \item[(1).] $\alpha(T_1,T_2)=\alpha(T_2,T_1)$ for any $T_1,T_2\in  \widetilde{\Delta}_{\widetilde{\lambda}}$;
    \item[(2).] $\alpha(T_1,T_2)=\alpha(T_1,T_3)+\alpha(T_3,T_2)$ if $T_1,T_2,T_3\in  \widetilde{\Delta}_{\widetilde{\lambda}}$ and $T_3$ separates $T_1$ and $T_2$, i.e.,
    $T_1$ and $T_2$ are in the different components of the closure of $\widetilde{S}-T_3$;
    \item[(3).] $\alpha$ is $\pi_1(S)$-invariant.
\end{itemize}
\end{definition}

Bonahon \cite{Bonahon} showed that any pleated surface $(\widetilde{f},\rho)$ with topological type $S$ and pleating locus $\lambda$ can be associated to a transverse cocycle $\alpha$, so that for two plaques $T_1$ and $T_2$, $\alpha(T_1,T_2)$ describes how `bent' $\widetilde{f}(T_2)$ is relative to $\widetilde{f}(T_1)$ in $\mathbb{H}^3$. If $T_1$ and $T_2$ share a common edge in $\widetilde{\lambda}$, $\alpha(T_1,T_2)$ is the external angle between $\widetilde{f}(T_1)$ and $\widetilde{f}(T_2)$. If $T_1$ and $T_2$ are separated by finite many plagues, $\alpha(T_1,T_2)$ is given by condition (2) in Definition \ref{transversecocycle}. See \cite{Bonahon} Section 7 for details. We refer to $\alpha$ as \emph{the bending cocycle for $(\widetilde{f},\rho)$}.\\

On the other hand, Bonahon \cite{Bonahon} shows that every transverse cocycle for $\lambda$ is the bending cocycle for some pleated surface $(\widetilde{f},\rho)$ with topological type $S$ and pleating locus $\lambda$.\\

In particular, let $\rho_0$ denote the Fuchsian representation $\rho_0:\pi_1(S)\rightarrow \mathrm{PSL}(2,\mathbb{R})\subset \mathrm{PSL}(2,\mathbb{C})$ and let $i:\mathbb{H}^2\rightarrow \mathbb{H}^3$ be the canonical embedding. Then $(i,\rho_0)$ is a pleated surface with topological type $S$ and pleating locus $\lambda$. The bending cocycle for $(i,\rho_0)$ is $0$.

\begin{theorem}\label{MMMZthm}
    Let $(\widetilde{f}_\alpha,\rho_\alpha)$ denote the pleated surface with topological type $S$ and pleating locus $\lambda$ associated to the transverse cocycle $\alpha$ for $\lambda$. Then $\rho_\alpha$ is $1$-Anosov over $S_{P_\lambda}$, which is uniform for any $\alpha$, i.e., the constants in Definition \ref{DefAnosovoverSp} is independent to the choice of $\alpha$.
\end{theorem}

This result is proved in an upcoming paper by Maloni--Martone--Mazzoli--Zhang \cite{MMMZ}, who use it to generalize the notion of pleated surfaces to higher rank Lie groups.\\

\bibliographystyle{amsplain}
\bibliography{reference.bib}

\end{document}